\documentclass[11pt]{amsart}

\usepackage{mathrsfs}
\usepackage{amsmath, amsthm, amsfonts}
\usepackage[percent]{overpic}

\usepackage{color}
\usepackage{MnSymbol}
\usepackage{fullpage}
\usepackage{bbm}

\usepackage{thmtools}
\usepackage{thm-restate}

\usepackage{hyperref}

\usepackage{cleveref}

\declaretheorem[name=Theorem,numberwithin=section]{thm}

\newtheorem*{thm*}{Theorem}
\newtheorem{cor}[thm]{Corollary}
\newtheorem{prop}[thm]{Proposition}
\newtheorem{lem}[thm]{Lemma}

\theoremstyle{definition}
\newtheorem{defn}[thm]{Definition}
\newtheorem*{ack}{Acknowledgements}
\newtheorem{ex}[thm]{Example}
\theoremstyle{remark}
\newtheorem{rem}[thm]{Remark}

\newcommand{\QQ}{\mathbb{Q}}

\newcommand{\HH}{\mathcal{H}}
\newcommand{\QQQ}{\mathcal{Q}}
\newcommand{\PPP}{\mathcal{P}}

\newcommand{\RR}{\mathbb{R}}
\newcommand{\CC}{\mathbb{C}}

\newcommand{\OO}{\mathcal{O}}
\newcommand{\ord}{\operatorname{ord}}

\newcommand{\ZZ}{\mathbb{Z}}

\newcommand{\PP}{\mathbb{P}}

\newcommand{\Pic}{\text{Pic}}

\newcommand{\Mgn}{\mathcal{M}_{g,n}}
\newcommand{\Mgb}{\overline{\mathcal{M}}_g}
\newcommand{\Mgnb}{\overline{\mathcal{M}}_{g,n}}

\newcommand{\Mbar}[2]{\overline{\mathcal{M}}_{{#1}, {#2}}}

\newcommand{\res}{\text{res}}

\newcommand{\SL}{\mbox{SL}}

\newtheorem{Thm*}{Theorem*}
\theoremstyle{definition}

\title{$k$-differentials on curves and rigid cycles in moduli space}
\author{Scott Mullane}
\date{\today}

\AtEndDocument{\bigskip{\footnotesize%
\textsc{Scott Mullane, Institut f\"{u}r Mathematik, Goethe-Universit\"{a}t Frankfurt, Robert-Mayer-Str. 6-8, 60325 Frankfurt am Main, Germany}\par
\textit{E-mail address}: \texttt{mullane@math.uni-frankfurt.de} \par
}}

\begin{document}
\thispagestyle{empty}

\maketitle

\begin{abstract}
For $g\geq2$, $j=1,\dots,g$ and $n\geq g+j$ we exhibit infinitely many new rigid and extremal effective codimension $j$ cycles in $\overline{\mathcal{M}}_{g,n}$ from the strata of quadratic differentials and projections of these strata under forgetful morphisms and show the same holds for $k$-differentials with $k\geq 3$ if the strata are irreducible. We compute the class of the divisors in the case of quadratic differentials which contain the first known examples of effective divisors on $\Mgnb$ with negative $\psi_i$ coefficients. 
\end{abstract}

\setcounter{tocdepth}{1}

\tableofcontents

\section{Introduction}
The effective cone of divisors in a projective variety broadly governs the birational geometry of the variety. For this reason, the structure of the effective cone of $\Mgnb$ has become a key question in the birational geometry of the moduli spaces of curves~\cite{HarrisMumford,EisenbudHarrisKodaira,Farkas23,FarkasPopa,Logan,ChenCoskun,CastravetTevelevHypertree}. Recently, there has also been growing interest in understanding the structure of the effective cones of higher codimension cycles that dictate the finer aspects of the birational geometry of these spaces~\cite{ChenCoskunH,FL1,FL2}. An initial task in this pursuit is to identify, in each codimension, the extremal effective cycles that span the boundary rays of the effective cone of cycles.

The moduli space of abelian differentials $\mathcal{H}(\mu)$ consists of pairs $(C,\omega)$ of a holomorphic or meromorphic one form $\omega$ on a smooth curve $C$ with fixed multiplicities of zeros and poles given by $\mu$. 
Recent seminal work exposes the fundamental algebraic attributes of these spaces~\cite{McMullen,KontsevichZorich,MoellerHodge,EskinMirzakhaniMohammadi,EskinMirzakhani,Filip}. 
Furthermore, the condition of the existence of a holomorphic or meromorphic one form of fixed signature has been used previously both explicitly and under many guises to obtain divisor classes~\cite{Cukierman,Diaz,ChenStrata,ChenCoskun,Logan,Muller,GruZak,FarkasBN,FarkasVerraTheta} and in lower genus, higher codimension cycles~\cite{ChenCoskunH,ChenTarasca,Blankers} in $\Mgb$ and $\Mgnb$. In contrast, conditions from $k$-differentials for $k\geq2$ have remained an untapped source for effective cycles and questions of the relation of these strata to the birational geometry of $\Mgnb$ have remained largely unexamined. The relevance of the strata of $k$-differentials for $k\geq 2$ to the birational geometry of $\Mgnb$ are as deserved of investigation as the $k=1$ case and in this paper we initiate this work.

For fixed $g$ and $n$ with $g\geq 2$, codimension $\leq n-g$, the author has previously exhibited, from the strata of meromorphic one forms, infinitely many extremal divisors~\cite{Mullane3} and higher codimension cycles~\cite{Mullane4} in $\Mgnb$  intersecting the interior of the moduli space. It is natural to ask if these and the other finitely many known extremal cycles\footnote{See the Introduction of~\cite{Mullane3, Mullane4} for a summary of the known extremal cycles in the cases of divisors and higher codimension cycles respectively.} of this type give all such extremal rays of the effective cones for fixed $g$ and $n$.

The stratum of $k$-canonical divisors for $k\geq 1$ with signature $\mu=(m_1,...,m_n)$, an integer partition of $k(2g-2)$ forms a subvariety of $\mathcal{M}_{g,n}$,
\begin{equation*}
\PPP^k(\mu):=\{[C,p_1,...,p_n]\in {\mathcal{M}}_{g,n}   \hspace{0.15cm}| \hspace{0.15cm}m_1p_1+...+m_np_n\sim kK_C\},
\end{equation*}
where the index $k$ is often dropped in the case that $k=1$ and for $k\geq 2$ if $k|m_i$ and $m_i\geq 0$ for all $i$, we impose the extra condition that 
\begin{equation*}
\frac{m_1}{k}p_1+\dots \frac{m_n}{k}p_n\nsim K_C
\end{equation*}
to omit the higher dimensional components coming from $k$-differentials that are $k$th powers of holomorphic differentials. The global $k$-residue condition that compactifies the strata of $k$-canonical divisors was given by~\cite{BCGGM2} extending their compactification of the $k=1$ case of the strata of canonical divisors~\cite{BCGGM}. 

Pushing forward under the forgetful morphism we obtain lower codimension cycles. In the codimension one case we define the closure as the divisor
\begin{equation*}
D^{n,k}_\mu=\overline{\{[C,p_1,...,p_n]\in\Mgn\hspace{0.15cm}|\hspace{0.15cm} \exists p_{n+1},\dots,p_{n+g-1} \text{ with }[C,p_1,...,p_{n+g-1}]\in\PPP^k(\mu)  \}},
\end{equation*}
in $\Mgnb$ for $\mu=(m_1,...,m_{n+g-1})$ with $\sum m_i=k(2g-2)$. In the holomorphic case where $k=1$ and all $m_i>0$ due to the change in dimension we obtain the divisors
\begin{equation*}
D^{n}_\mu=\overline{\{[C,p_1,...,p_n]\in\Mgn\hspace{0.15cm}|\hspace{0.15cm} \exists p_{n+1},\dots,p_{n+g-2} \text{ with }[C,p_1,...,p_{n+g-2}]\in\PPP(\mu)  \}},
\end{equation*}
in $\Mgnb$ for $\mu=(m_1,...,m_{n+g-2})$ and $m_i>0$ with $\sum m_i=2g-2$. 

In a projective variety $X$, a moving curve is curve class $B$ with $B\cdot D\geq 0$ for any pseudo-effective divisor $D$. Taking a fibration of $\overline{\PPP}^k(\mu)$, the closure of $\PPP^k(\mu)$, with $|\mu|\geq g+1$ we obtain curves in $\Mgnb$ 
\begin{equation*}
B^{n,k}_{\mu}:=\overline{\bigl\{ [C,p_1,...,p_n]\in\mathcal{M}_{g,n} \hspace{0.15cm}\big| \hspace{0.15cm} \text{fixed general $[C,p_{g+2},...,p_m]\in\mathcal{M}_{g,m-g-1}$ and $[C,p_1,\dots,p_m]\in {\PPP}^k(\mu)$}    \bigr\}   }.
\end{equation*}
For $m=|\mu|\geq n+g$ these curves provide covering curves for $\Mgnb$ as irreducible curves with class proportional to $B^{n,k}_{\mu}$ cover a Zariski dense subset of $\Mgnb$ and hence must have non-negative intersection with every pseudo-effective divisor. For any divisor $D^{n,k}_\mu$, with $|\mu|\geq 2g$ in Theorem~\ref{thm:boundary} we show 
\begin{equation*}
B^{n,k}_{\mu,1,-1}\cdot D^{n,k}_{\mu}=0,
\end{equation*}
placing $D^{n,k}_{\mu}$ on the boundary of the pseudo-effective cone of divisors.

A covering curve $B$ of an effective divisor $D$ is a curve class such that irreducible curves with numerical class equal to $B$ cover a Zariski dense subset of $D$. If $B\cdot D<0$ and $D$ is irreducible then $D$ is rigid and extremal in the pseudo-effective cone (Lemma~\ref{lemma:covering}). For certain signatures with $|\mu|=2g$ we are able to show the covering curve $B^{n,k}_\mu$ for the divisor $D^{n,k}_\mu$ is a component of a specialisation of $B^{n,k}_{\mu,1,-1}$. The other component of the curve is completely contained in the boundary of $\Mgnb$ and the positive intersection of this boundary curve with $D^{n,k}_\mu$ gives 
\begin{equation*}
B^{n,k}_\mu\cdot D^{n,k}_\mu<0,
\end{equation*}
showing these divisors to be rigid and extremal if $D^{n,k}_\mu$ is irreducible.

\begin{restatable}{thm}{extremal}
\label{thm:extremal}
The divisors $D^{g+1,k}_{\underline{d},k^{g-1}}$ in $\overline{\mathcal{M}}_{g,g+1}$ for $g\geq 2$, $k\geq 2$ are rigid and extremal for $\underline{d}=(d_1,d_2,d_3,k^{g-2})$ with $d_1+d_2+d_3=k$, and $d_i\ne 0$ if $D^{g+1,k}_{\underline{d},k^{g-1}}$ is irreducible. 
\end{restatable}

The classification of the irreducible connected components of $\PPP^k(\mu)$ for meromorphic signature $\mu$ remains in progress for $k\geq 3$, however the classification in the case of quadratic differentials~\cite{Lanneau, ChenGendron} yields the following corollary.

\begin{restatable}{cor}{quaddiv}
\label{cor:quaddiv}
The divisors $D^{g+1,2}_{\underline{d},2^{g-1}}=Q^{g+1}_{\underline{d},2^{g-1}}$ in $\overline{\mathcal{M}}_{g,g+1}$ for $g\geq 2$ are rigid and extremal for $\underline{d}=(d_1,d_2,d_3,2^{g-2})$ with $d_1+d_2+d_3=2$, $d_i\ne 0$ and some $d_i$ odd. 
\end{restatable}

This corollary provides infinitely many new rigid and extremal pseudo-effective divisors in $\Mgnb$ for $g\geq 2, n\geq g+1$ and hence also an alternate proof that the pseudo-effective cone is not rational polyhedral and hence these are not Mori dream spaces. Further, the completion of the classification of the connected components of the strata of meromorphic $k$-differentials is expected to yield infinitely many more rigid and extremal cycles for each $k\geq 3$.

We proceed to obtain a general formula for the class of these divisors. The Picard variety method enumerates certain instances of points satisfying equations in the Jacobian of components of a nodal curve, while the the global $k$-residue condition gives the condition that the associated twisted canonical divisors are smoothable. Utilising the symmetries in the class of the divisor $Q^{2g-2}_{1^{2g-2},2^{g-1}}$ in $\Mbar{g}{2g-2}$ due to the symmetries in the signature we compute the class of this divisor by comparing the intersection of this divisor and the intersection of the standard basis of $\Pic_\QQ(\Mbar{g}{2g-2})$ with a number of test curves in the boundary of the moduli space.

\begin{restatable}{prop}{Qg}
\label{prop:Qg}
The class of the divisor $Q_g=Q^{2g-2}_{1^{2g-2},2^{g-1}}$ in $\Pic_\QQ(\overline{\mathcal{M}}_{g,2g-2})$ is:
\begin{equation*}
Q_g=3(2^{2g-3})\sum_{j=1}^{2g-2}\psi_j-4^g\lambda+4^{g-2}\delta_0-2^{2g-3}\sum_{S,S^c\ne \emptyset} (|S|-2i)(|S|-2i+2)\delta_{i:S}-\sum_{i=1}^g 2^{2(g-i)-1}(4^i(i-1)+2)i\delta_{i:\emptyset}. 
\end{equation*}
\end{restatable}

By carefully identifying the components and multiplicities of the pullback of divisors $Q_{\underline{d},2^{g-1}}^n$ under gluing morphisms we obtain a general formula for the class of these divisors from the class of $Q_g$ and the known classes in the case of holomorphic abelian differentials.

\begin{restatable}{prop}{Qd}
\label{prop:Qd}
The class of the divisor $Q_{\underline{d},2^{g-1}}^n$ in $\overline{\mathcal{M}}_{g,n}$ for $\underline{d}=(d_1,\dots,d_n)$ with $\sum d_i=2g-2$ is
\begin{equation*}
(4^g-1)\sum_{j=1}^n\frac{(d_j+2)d_j}{8}\psi_j-(4^g-1)\lambda+4^{g-2}\delta_0-\sum_{d_S\geq 2i}\frac{1}{8}(d_S-2i+2)(4(4^i-1)+(d_S-2i)(4^g-1))\delta_{i:S}
\end{equation*}
for all $d_i$ even and non-negative and
\begin{equation*}
2^{2g-3}\sum_{j=1}^{n} d_j(d_j+2)\psi_j-4^g\lambda+4^{g-2}\delta_0+\sum_{i,S}c_{i:S}\delta_{i:S}
\end{equation*}
where
\begin{equation*}
c_{i:S}=\begin{cases}
-(d_S-2i+2)(2^{2g-3}(d_S-2i)+2^{2i-1})&\text{ if $N\subset S$ and $d_S\geq 2i$}\\
-2^{2g-3}(d_S-2i)(d_S-2i+2)& \text{ if $N\subset S$ and $d_S<2i$ or $N\nsubset S, S^c$}
\end{cases}
\end{equation*}
otherwise, where $N=\{j\text{ } |\text{ } d_j \text{ is odd or negative}\}$.
\end{restatable}
It is well known that effective divisors in $\overline{\mathcal{M}}_{g,1}$ must have non-negative $\psi$ coefficient. However, for a divisor above in $\Mgnb$ with $n\geq 2$, setting some $d_i=-1$ gives $c_{\psi_i}=-2^{2g-3}<0$ providing the first known examples of effective divisor classes with negative $\psi_i$ coefficients in $\Mgnb$. Hence in addition to the extremal divisors of Corollary~\ref{cor:quaddiv}, the above divisors for $n\geq 2$ extend the known effective cone. In particular, this shows that the effective cone of $\Mbar{2}{2}$ is larger than the cone generated by the seven known extremal rays\footnote{The known extremal rays are the four irreducible boundary divisors, the two pullbacks of the Weierstrass divisor on $\Mbar{2}{1}$ under the morphism forgetting one of the marked points, and the closure of the locus of smooth pointed curves where the points are conjugate under the hyperelliptic involution.}, a question that has persisted for some time~\cite{Rulla}.

At this point we can use the strategy of gluing morphisms by Chen-Coskun~\cite{ChenCoskunH} in the genus $g=1$ case, generalised to the general genus $g\geq 2$ case in~\cite{Mullane4}, to construct for each rigid and extremal divisor in $\Mgnb$, an extremal effective codimension $2$ cycle in $\overline{\mathcal{M}}_{g+1,n-1}$. However, all such cycles are supported in the boundary and reflect the structure of the effective cone of divisors on $\Mgnb$. For this reason we restrict our attention to cycles that intersect the interior of the moduli space. 

Our method to prove the rigidity and extremality of cycles in higher codimension utilises the rigidity of the divisor classes as the base case of an inductive argument similar to that of Chen and Tarasca~\cite{ChenTarasca}. Consider any effective decomposition of the codimesion $d$ cycle $[V]$ for irreducible $V$ in $\Mgnb$ given by
$$ [V]= \sum c_i[V_i]$$
with $c_i>0$ and $V_i$ irreducible codimension $d$ subvarieties of $\Mgnb$ not supported on $V$. We show by pushing forward this relation under the morphism forgetting the last point $\varphi:\Mgnb\longrightarrow \Mbar{g}{n-1}$, that in our case the rigidity and irreducibility of the codimension $d-1$ cycle $\varphi_*[V]=[\varphi_*V]$ implies that some $V_i$ must be supported on $V$ providing a contradiction to the assumed effective decomposition and hence showing $[V]$ is rigid and extremal. In low genus $g=2,3$ some extra candidates for effective decompositions arise which we eliminate by pushing forward an implied effective cycle under a forgetful morphism to obtain an assumed effective divisor which has negative intersection with a moving curve constructed in Proposition~\ref{moving} providing a contradiction.

\begin{restatable}{thm}{fullextremal}
\label{fullextremal}
For $g\geq 2$, $k\geq 1$ and $j=0,\dots,g-1$ if $\PPP^k(d_1,d_2,d_3,k^{2g-3})$ is irreducible the cycle $[\varphi_j{}_*\overline{\PPP}^k(d_1,d_2,d_3,k^{2g-3})]$ is rigid and extremal in $\mbox{Eff}^{\hspace{0.1cm}g-j}(\overline{\mathcal{M}}_{g,2g-j})$, where $\varphi_j:\overline{\mathcal{M}}_{g,2g}\longrightarrow\overline{\mathcal{M}}_{g,2g-j}$ forgets the last $j$ points with $d_1+d_2+d_3=k$ and some $d_i=k$ if $g=2$ and $j=0$.
\end{restatable}

Again, though the classification of the connected components of meromorphic strata of $k$-differentials is not complete, the classification in the case of quadratic differentials~\cite{Lanneau, ChenGendron} yields the following corollary.

\begin{restatable}{cor}{Quadfullextremal}
\label{Quadfullextremal}
For $g\geq 2$ and $j=0,\dots,g-1$ the cycle $[\varphi_j{}_*\overline{\mathcal{Q}}(d_1,d_2,d_3,2^{2g-3})]$ is rigid and extremal in $\mbox{Eff}^{\hspace{0.1cm}g-j}(\overline{\mathcal{M}}_{g,2g-j})$, where $\varphi_j:\overline{\mathcal{M}}_{g,2g}\longrightarrow\overline{\mathcal{M}}_{g,2g-j}$ forgets the last $j$ points with $d_i\ne 0$, $d_1+d_2+d_3=2$, some $d_i$ odd and in the case $g=2$ and $j=0$ some $d_i=2$.
\end{restatable}

This provides infinitely many new rigid and extremal rays of the effective cone in these cases and it is expected that the future classification of the irreducible connected components of the strata of meromorphic $k$-differentials for $k\geq 3$ will yield more infinite families of extremal cycles.

\begin{ack}
The author was supported by the Alexander von Humboldt Foundation during the preparation of this article.
\end{ack}

\section{Preliminaries}
\subsection{Strata of $k$-differentials}
The \emph{stratum of $k$-differentials of signature $\mu=(m_1,\dots,m_n)$}, an integer partition of $k(2g-2)$ is defined as
\begin{equation*}
\HH^k(\mu):=\{ (C,\omega) \hspace{0.15cm}|\hspace{0.15cm} g(C)=g,\hspace{0.05cm} (\omega)=m_1p_1+...+m_np_n, \text{ for $p_i$ distinct}\}
\end{equation*}
where $\omega$ is a meromorphic differential on $C$ which for $k\geq 2$ is not equal to the $k$th power of a holomorphic differential. Hence $\HH^k(\mu)$ is the space of $k$-differentials with prescribed multiplicities of zeros and poles by $\mu$. If non-empty, $\HH^k(\mu)$ has dimension $2g+n-2$ unless $k=1$ and $\mu$ is holomorphic (all $m_i\geq 0$) where the dimension is $2g+n-1$.

The \emph{stratum of $k$-canonical divisors with signature $\mu$} is defined as 
\begin{equation*}
\PPP^k(\mu):=\{[C,p_1,...,p_n]\in \Mgn   \hspace{0.15cm}| \hspace{0.15cm}m_1p_1+...+m_np_n\sim K_C^k\}
\end{equation*}  
where again, in the case that $k|m_i$ and $m_i\geq 0$ for all $i$ we impose the additional requirement
\begin{equation*}
\frac{m_1}{k}p_1+...+\frac{m_n}{k}p_n\nsim K_C.
\end{equation*}
The codimension of non-empty $\PPP^k(\mu)$ in $\Mgn$ is equal to $g-1$ if $k=1$ and $\mu$ is holomorphic and equal to $g$ in all other cases.

\subsection{Connected components of the strata of quadratic differentials}

The classification of the connected components of the strata of quadratic differentials of finite area was completed by Lanneau~\cite{Lanneau}.
\begin{thm}
Suppose $\QQQ(\mu):=\HH^2(\mu)$ is a stratum of quadratic differentials with genus $g\geq 2$ and no poles of order greater than one. Then the following statements hold:
\begin{enumerate}
\item[(1)]
In $g=2$ there are two non-connected strata $\QQQ(-1,-1,6)$ and $\QQQ(-1,-1,3,3)$ with two connected components. All other strata are connected.
\item[(2)]
In genus $g\geq3$ there are three families of strata with two connected components
\begin{eqnarray*}\begin{array}{rlrr}
 \QQQ(4(g-k)-6,4k+2)   &\text{ for }0\leq k\leq g-2,\\
  \QQQ(2(g-k)-3,2(g-k)-3,4k+2)   &\text{ for }0\leq k\leq g-1,\\
   \QQQ(2(g-k)-3,2(g-k)-3,2k+1,2k+1)   &\text{ for }-1\leq k\leq g-2.
 \end{array}
\end{eqnarray*}
each stratum possessing a hyperelliptic and non-hyperelliptic component. There are $4$ sporadic strata in $g=3$ and $4$:
$$\QQQ(-1,9), \hspace{0.5cm}\QQQ(-1,3,6), \hspace{0.5cm}\QQQ(-1,3,3,3), \hspace{0.5cm}\QQQ(12)$$
which each have two connected components. All other stratum are non-empty and connected.
\end{enumerate}
\end{thm}
\begin{proof}\cite{Lanneau}\end{proof}
 
Although the classification of the connected components for higher strata of $k$-differentials is still in progress~\cite{ChenGendron}, the classification for quadratic differentials with infinite area (a non-simple pole) is complete. In addition to the components that arise as the square of a connected component of meromophic one-forms classified by Boissy~\cite{Boissy}, Chen and Gendron classified the primitive connected components.

\begin{thm}
Suppose $\QQQ(\mu):=\HH^2(\mu)$ is a stratum of quadratic differentials with genus $g\geq 2$ and at least one pole of order two or more. Then there are two primitive connected components of quadratic differentials if 
\begin{itemize}
\item
$\mu=(2n,-l,-l)$ with $l$ odd, 
\item
$\mu=(n,n,-2l)$ with $n$ odd, 
\item
$\mu=(n,n,-l,-l)$ with $n$ and $l$ not both even, 
\item
$\mu=(2n,-2)$ or $(2n,2n,-2)$. 
\end{itemize}
In all other cases there is exactly one primitive connected component.
\end{thm}


\begin{proof}\cite{ChenGendron} 
\end{proof}

Observe that as a consequence of the above two theorems we have the following.
\begin{cor}\label{irredQ}
$\QQQ(\mu)$ is connected and hence irreducible for $\mu=( d_1,d_2,d_3,2^{2g-3})$ when $d_i\ne 0$ and some $d_i$ is odd.
\end{cor}

\subsection{Degeneration of $k$-differentials}
A stable pointed curve $[C,p_1,...,p_n]\in\Mgnb$ is contained in \emph{the stratum of twisted $k$-canonical divisors  with signature $\mu$} defined by Farkas and Pandharipande~\cite{FP} in the $k=1$ case, denoted $\widetilde{\PPP}^k(\mu)$, if there exists a collection of (possibly meromorphic) non-zero $k$-differentials $\{\eta_j\}$ on the irreducible components $C_j$ of $C$ known as a \emph{twisted $k$-differential} with $(\eta_j)= D_j\sim K_{C_j}^k$ such that
\begin{enumerate}
\item
The support of $D_j$ contains the set of marked points and the nodes lying in $C_j$, moreover if $p_i\in C_j$ then $\ord_{p_i}(D_j)=m_i$.
\item
If $q$ is a node of $C$ and $q\in C_i\cap C_j$ then $\ord_q(D_i)+\ord_q(D_j)=-2k$.
\item
If $q$ is a node of $C$ and $q\in C_i\cap C_j$ such that $\ord_q(D_i)=\ord_q(D_j)=-k$ then for any $q'\in C_i\cap C_j$, we have $\ord_{q'}(D_i)=\ord_{q'}(D_j)=-k$. We write $C_i\sim C_j$.
\item
If $q$ is a node of $C$ and $q\in C_i\cap C_j$ such that $\ord_q(D_i)>\ord_q(D_j)$ then for any $q'\in C_i\cap C_j$ we have $\ord_{q'}(D_i)>\ord_{q'}(D_j)$. We write $C_i\succ C_j$.
\item
There does not exist a directed loop $C_1\succeq C_2\succeq...\succeq C_k\succeq C_1$ unless all $\succeq$ are $\sim$.
\end{enumerate}  
In addition to $\overline{\PPP}^k(\mu)$ known as \emph{the main component},  $\widetilde{\PPP}^k(\mu)$ contains extra components completely contained in the boundary. Bainbridge, Chen, Gendron, Grushevsky and M\"oller provided the condition that a twisted $k$-canonical divisor lies in the main component. First consider the case $k=1$~\cite{BCGGM}. Let $\Gamma$ be the dual graph of $C$. A twisted canonical divisor of type $\mu$ is the limit of canonical divisors on smooth curves if there exists a twisted differential $\{\eta_j\}$ on $C$ such that
\begin{enumerate}
\item
If $q$ is a node of $C$ and $q\in C_i\cap C_j$ such that $\ord_q(D_i)=\ord_q(D_j)=-1$ then $\res_q(\eta_i)+\res_q(\eta_j)=0$; and
\item
there exists a full order on the dual graph $\Gamma$, written as a level graph $\overline{\Gamma}$, agreeing with the order of $\sim$ and $\succ$, such that for any level $L$ and any connected component $Y$ of  $\overline{\Gamma}_{>L}$ that does not contain a prescribed pole we have
\begin{equation*}
\sum_{\begin{array}{cc}\text{level}(q)=L, \\q\in C_i\subset Y\end{array}}\res_{q}(\eta_i)=0
\end{equation*} 
\end{enumerate}
Condition $(b)$ is known as the \emph{global residue condition}. For $k\geq 2$, twisted $k$-differentials can be lifted by a covering construction and the conditions for membership in the main component reduced to the conditions in the case of $k=1$ on the covering curve~\cite{BCGGM2}.  Bainbridge, Chen, Gendron, Grushevsky and M\"oller show via this strategy that a twisted $k$-canonical divisor of type $\mu$ is the limit of $k$-canonical divisors on smooth curves if there exists a twisted $k$-differential $\{\eta_j\}$ on $C$ such that
\begin{enumerate}
\item
If $q$ is a node of $C$ and $q\in C_i\cap C_j$ such that $\ord_q(D_i)=\ord_q(D_j)=-k$ then $\res^k_q(\eta_i)=(-1)^k\res^k_q(\eta_j)$; and
\item
there exists a full order on the dual graph $\Gamma$, written as a level graph $\overline{\Gamma}$, agreeing with the order of $\sim$ and $\succ$, such that for any level $L$ and any connected component $Y$ of  $\overline{\Gamma}_{>L}$ one of the following cases holds
\begin{enumerate}
\item 
$Y$ contains a marked pole.
\item
$Y$ contains a vertex $v$ such that $\eta_v$ is not a $k$-th power of a (possibly meromorphic)
abelian differential.
\item
("Horizontal criss-cross in $Y$ .") For every vertex $v$ of $Y$ the $k$-differential $\eta_v$ is the $k$-th
power of an abelian differential $\omega_v$. Moreover, for every choice of a collection of $k$-th
roots of unity $\{\zeta_v | v\in Y\}$ there exists a horizontal edge $e$ in $Y$ where the differentials
$\{\zeta_v\omega_v\}_{v\in Y}$ do not satisfy the matching residue condition.
\item
("Vertical criss-cross in $Y$ .") For every vertex $v$ of $Y$ the $k$-differential $\eta_v$ is the $k$-th
power of an abelian differential $\omega_v$. Moreover, there exists a level $K > L$ and a collection
of $k$-th roots of unity $\{\zeta_e | e\in E\}$ indexed by the set $E$ of non-horizontal edges $e$ of $Y$
whose lower end lies in $Y_{=K}$, such that the following two conditions hold. First, there
exists a directed simple loop $\gamma$ in the dual graph of $Y_{\geq K}$ such that
\begin{equation*}
\prod_{e\in \gamma\cap E}\zeta_{e}^{\pm 1}\ne 1,
\end{equation*}
where the sign of the exponent is $\pm 1$ according to whether $\gamma$ passes through $e$ in upward or downward direction. Second, for every connected component $T$ of $Y_{>K}$ the equation
\begin{equation*}
\sum_{e\in E_T} \zeta_e\res_{q_e^-}(\omega_{v^-(e)})=0
\end{equation*}
holds, where $E_T$ is the subset of edges in $E$ such that their top vertices lie in $T$.
\item
("$Y$ imposes a residue condition.") The $k$-residues at the edges $e_1,\dots, e_N$ joining $Y$ to
$\Gamma_{=L}$ satisfy the equation
\begin{equation*}
P_{N,k}\left(\res^k_{q_{e_1}^-}(\eta_{v^-_{e_1}}),\dots,\res^k_{q_{e_N}^-}(\eta_{v^-_{e_N}})\right)=0
\end{equation*}
for
\begin{equation*}
P_{n,k}(R_1,\dots,R_n):=\prod_{\{(r_1,\dots,r_n)|r_i^k=R_i  \}}\sum_{i=1}^n r_i,
\end{equation*}
where the product is taken over all $n$-tuples of complex numbers $\{r_1,\dots,r_n\}$ such that $r_i^k=R_i$ for all $i$. As $P_{n,k}$ is symmetric with respect to the $k$th roots of $R_i$ it is a polynomial in $R_i$.
\end{enumerate}
\end{enumerate}
Condition $(e)$ is known as the \emph{global $k$-residue condition}.

\begin{ex}\label{Ex1} Fix $k\geq 2$ and $(d_1,d_2,d_3)\in \ZZ^3$ such that $d_1+d_2+d_3=k$. Consider $k$-differentials $\omega_C$ on a genus $g=3$ curve $C$ and $\omega_Y$ on a rational curve $Y$ such that 
\begin{eqnarray*}
&&(\omega_C)=ky'+kp_4+kp_5+kp_6 \sim kK_X\\
&&(\omega_Y)=-3ky+d_1p_1+d_2p_2+d_3p_3\sim kK_Y
\end{eqnarray*}
and further assume that $\omega_C$ is the $k$th power of a holomorphic differential. Then the twisted $k$-differential $\{\omega_C,\omega_Y\}$ gives
\begin{equation*}
[C\cup_{y'=y}Y,p_1,\dots,p_6]\in\widetilde{\PPP}^k(d_1,d_2,d_3,k^3)\subset \overline{\mathcal{M}}_{3,6}
\end{equation*} 
pictured in Figure 1. But not all such stable pointed curves belong to the main component $\overline{\PPP}^k(d_1,d_2,d_3,k^3)$. The global $k$-residue condition gives the requirements that such a stable pointed curve is indeed smoothable. In this case any the twisted $k$-differential associated to $[C\cup_{y'=y}Y,p_1,\dots,p_6]$ implies a unique level graph pictured in Figure 1. Hence the global $k$-residue condition is simply
\begin{equation*}
\res^k_{y}(\omega_Y)=0.
\end{equation*}

\begin{figure}[htbp]
\begin{center}
\begin{overpic}[width=1\textwidth]{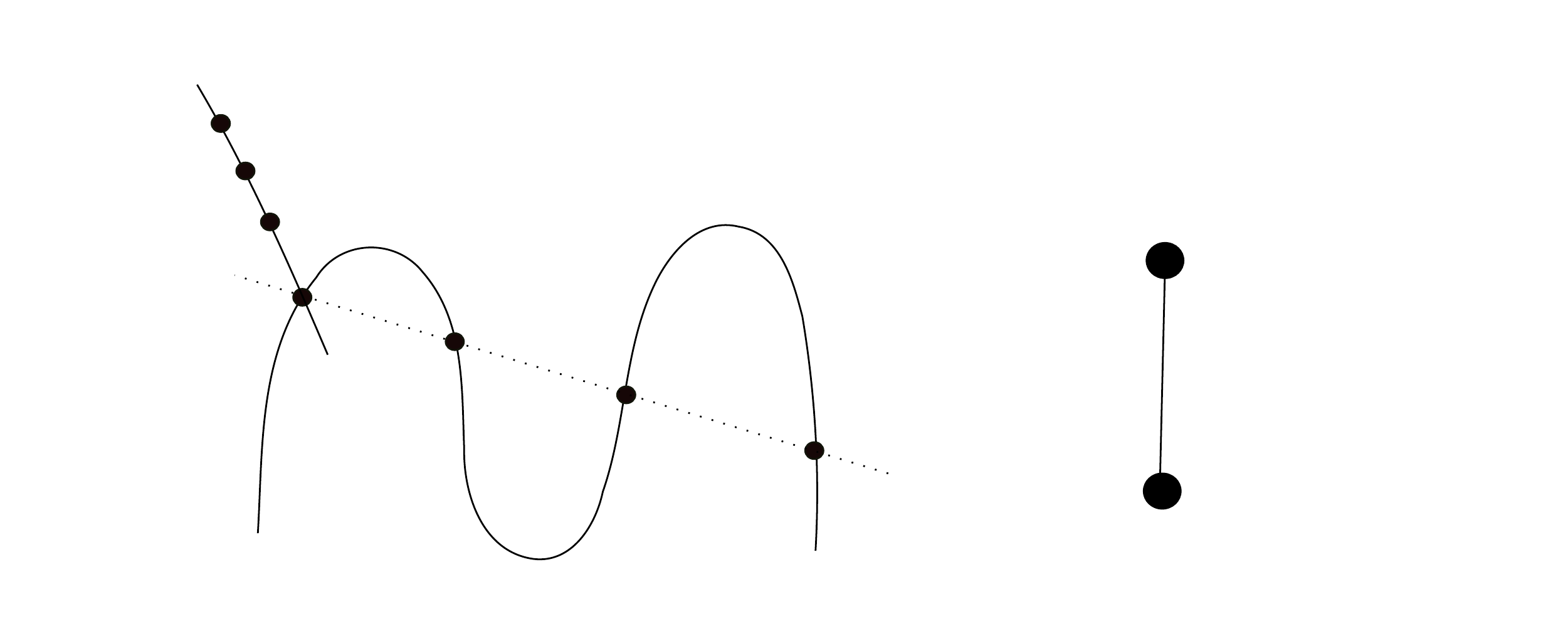}
\put (13,8){$C$}
\put (21,30){$Y\cong \PP^1$}
\put (15.5,33){$p_1$}
\put (17,30){$p_2$}
\put (18.5,26.5){$p_3$}
\put (16,20.5){$y$}
\put (21,21.5){$y'$}
\put (26,17){$p_4$}
\put (36.5,14){$p_5$}
\put (49,10){$p_6$}
\put (24,0){$y'+p_4+p_5+p_6\sim K_C$}
\put (73.75,27){$C$}
\put (73.5,5){$Y$}
\put (68,0){Level graph $\overline{\Gamma}$}
\put (13,-7){\footnotesize{Figure 1: A twisted $k$-differential in $\widetilde{\PPP}^k(d_1,d_2,d_3,k^3)\subset \overline{\mathcal{M}}_{3,6}$  } }
\end{overpic}
\vspace{0.6cm}
\end{center}
\end{figure}
\end{ex}

\begin{ex}\label{Ex2} Consider the unique (up to scaling) $k$-differential $\omega_X$ on rational curve $X$ with
\begin{equation*}
(\omega_X)=-2kx+p_7-p_8\sim kK_X.
\end{equation*} 
Then for any $x'\in C$ with $x'\ne y',p_4,p_5,p_6$, the twisted $k$-differential $\{\omega_C,\omega_X,\omega_Y\}$ gives
\begin{equation*}
[C\cup_{y'=y}Y\cup_{x'=x}X,p_1,\dots,p_8]\in\widetilde{\PPP}^k(d_1,d_2,d_3,k^3,1,-1)\subset \overline{\mathcal{M}}_{3,8}
\end{equation*} 
pictured in Figure 2. Again, not all such stable pointed curves belong to the main component $\overline{\PPP}^k(d_1,d_2,d_3,k^3)$. In this case the dual graph of the curve contains $3$ vertices and any possible twisted differential implies two conditions $C\succ X$ and $C\succ Y$ on the level dual graph $\overline{\Gamma}$. There are three level graphs that satisfy these conditions. However, if $X$ and $Y$ appear on different levels of the graph the global $k$-residue condition will imply that $\res^k_{Z}(\omega_Z)=0$ for the component $Z=X$ or $Y$ appearing in the higher level. But observe $\res^k_{x}(\omega_X)\ne 0$. Hence only level graphs with $X$ on the lowest level are possible. In the case that $Y$ appears on a higher level than $X$ the global $k$-residue gives the requirement
\begin{equation*} 
\res^k_y(\omega_Y)=0.
\end{equation*}
If $X$ and $Y$ appear on the same level in $\overline{\Gamma}$ the global $k$-residue condition gives
\begin{equation*}
\res^k_x(\omega_X)=(-1)^k\res^k_y(\omega_Y).
\end{equation*}
These two possible level graphs are pictured in Figure 2.

\begin{figure}[htbp]
\begin{center}
\begin{overpic}[width=1\textwidth]{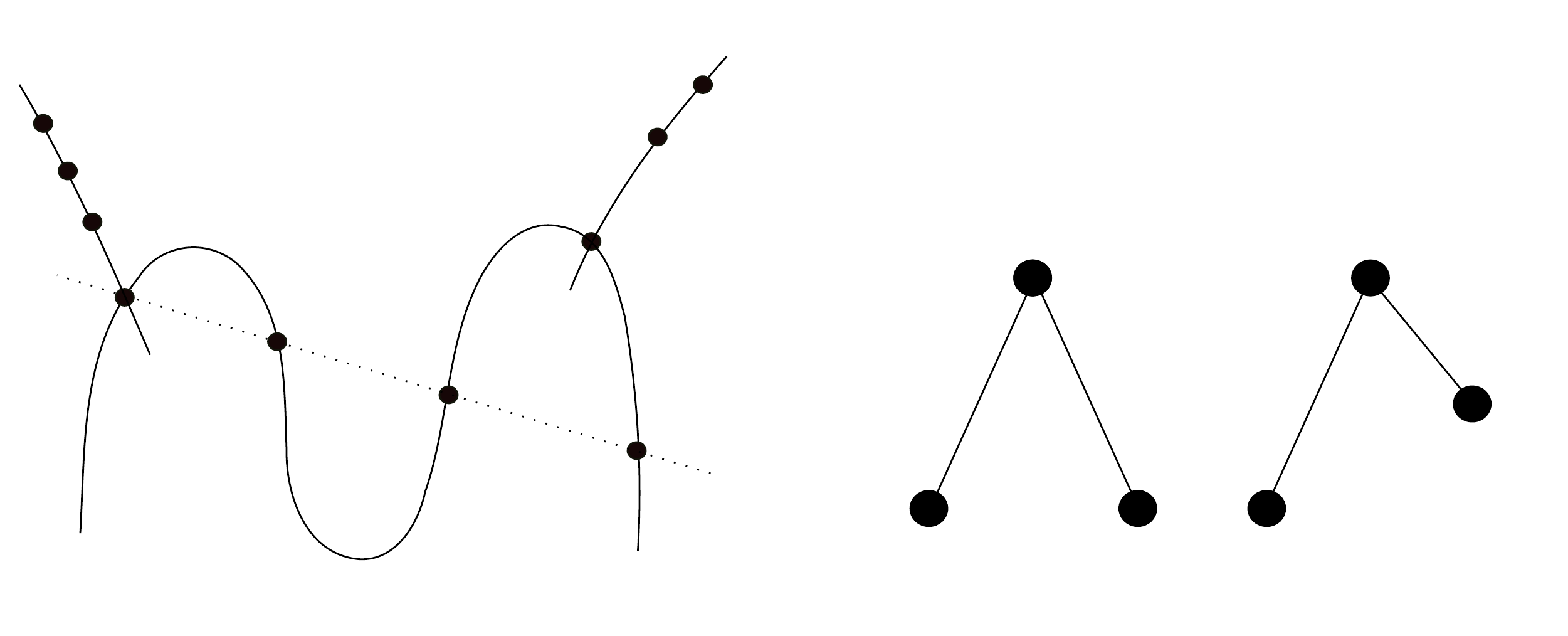}
\put (2,8){$C$}
\put (10,30){$Y\cong \PP^1$}
\put (4,33){$p_1$}
\put (5.5,30){$p_2$}
\put (7,26.5){$p_3$}
\put (5,20.5){$y$}
\put (10,21.5){$y'$}
\put (15,17){$p_4$}
\put (25.5,14){$p_5$}
\put (38,10){$p_6$}
\put (41,27){$X\cong \PP^1$}
\put (39,32){$p_7$}
\put (41.5,35){$p_8$}
\put (37,27){$x$}
\put (37,21){$x'$}
\put (13,0){$y'+p_4+p_5+p_6\sim K_C$}
\put (65,25){$C$}
\put (71.5,4){$Y$}
\put (58,4){$X$}
\put (86.5,25){$C$}
\put (93,10.5){$Y$}
\put (80,4){$X$}
\put (63,0){Possible level graphs $\overline{\Gamma}$}
\put (13,-7){\footnotesize{Figure 2: A twisted $k$-differential in $\widetilde{\PPP}^k(d_1,d_2,d_3,k^3,1,-1)\subset \overline{\mathcal{M}}_{3,8}$   } }
\end{overpic}
\vspace{0.6cm}
\end{center}
\end{figure}
\end{ex}

\subsection{Divisor theory on $\Mgnb$}
$\Pic_\QQ(\Mgn)$ is generated by $\lambda$, the first Chern class of the Hodge bundle, $\psi_i$ the first Chern class of the cotangent bundle on $\Mgnb$ associated with the $i$th marked point for $i=1,\dots,n$ and the irreducible components of the boundary $\Delta_0$ the locus of curves in $\Mgnb$ with a nonseparating node and $\Delta_{i:S}$ for $0\leq i\leq g$, $S\subset \{1,\dots,n\}$ the locus of curves with a separating node that separates the curve such that one of the components has genus $i$ and contains precisely the marked points in $S$. Let $\delta_0$ denote the class of $\Delta_0$ and $\delta_{i:S}$ the class of $\Delta_{i:S}$. Hence $\delta_{i:S}=\delta_{g-i:S^c}$ and observe we require $|S|\geq 2$ for $i=0$ and $|S|\leq n-2$ for $i=g$.

For $g\geq 3$, these divisors freely generate $\Pic_\QQ(\Mgnb)$, but for $g=2$, the classes $\lambda$, $\delta_0$ and $\delta_1$ generate $\Pic_\QQ(\overline{\mathcal{M}}_2)$ with the relation
\begin{equation*}
\lambda=\frac{1}{10}\delta_0+\frac{1}{5}\delta_1.
\end{equation*}
Pulling back this relation under the forgetful morphism forgetting all marked points gives the only relation on these generators in  $\Pic_\QQ(\overline{\mathcal{M}}_{2,n})$.

\subsection{Maps between moduli spaces}
For a fixed general $[X,q,q_1]\in \mathcal{M}_{h,2}$ consider the map 
\begin{eqnarray*}
\begin{array}{cccc}
\pi_h:&\overline{\mathcal{M}}_{g,n}&\rightarrow& \overline{\mathcal{M}}_{g+h,n}\\
&[C,p_1,\dots,p_n]&\mapsto&[C\bigcup_{p_1=q}X,q_1,p_2,\dots,p_n].
\end{array}
\end{eqnarray*} 
that glues points $p_1$ and $q$ to form a node. The pullback of the generators of $\Pic_\QQ(\overline{\mathcal{M}}_{g+h,n})$ are presented in~\cite{AC}
\begin{equation*}
\pi_h^*\lambda=\lambda, \hspace{0.5cm}\pi_h^*\delta_0=\delta_0,\hspace{0.5cm}\pi_h^*\delta_{h:\{1\}}=-\psi_1
\end{equation*}
and
\begin{equation*}
\pi_h^*\psi_i=\begin{cases}
0 &\text{ for i=1}\\
\psi_i &\text{ for $i=2,\dots,n$}
\end{cases}
\end{equation*}
and for $1\in S$
\begin{equation*}
\pi_h^*\delta_{i:S}=\begin{cases}
0 &\text{ for $i<h$}\\
-\psi_1 &\text{ for $i=h,S=\{1\}$}\\
\delta_{i-h:S}&\text{ otherwise.}\\
\end{cases}
\end{equation*}

\subsection{Divisors from $k$-strata of differentials}
The following divisor notation is used in this paper. 

\begin{defn}
For $|\mu|=g-2+n$ if $k=1$ and all entries of $\mu$ are $\geq 0$ and $|\mu|=g-1+n$ otherwise, if we express $\mu$ in exponential notation as
\begin{equation*}
\mu=(d_1,\dots,d_n,e_1^{\alpha_1},\dots,e_r^{\alpha_r}),
\end{equation*}
where $e_i\ne e_j$ for $i\ne j$.
Then $D^{n,k}_\mu$ for $n\geq 1$ is the divisor in $\Mgnb$ defined by:
\begin{equation*}
D^{n,k}_\mu:=\frac{1}{\alpha_1!\cdots\alpha_r!}\varphi_*\overline{\PPP}^k(\mu),
\end{equation*}
where $\varphi$ forgets the last $r=g-2$ points for $k=1$ and all $d_i,e_i\geq 0$ or the last $r=g-1$ marked points otherwise. 
\end{defn}
In the case that $k=1$ we will drop this index, that is, $D^{n,1}_\mu=D^{n}_\mu$. In the case of quadratic differentials, that is, $k=2$ we use the notation $D^{n,2}_\mu=Q^n_\mu$. For example, let $\underline{d}=(d_1,\dots,d_n)$ be an $n$-tuple of integers satisfying $\sum d_j=g$ with $d_j\geq 0$. Logan \cite{Logan} computed the class of the pointed Brill-Noether divisors in $\Mgnb$, which from our perspective are the divisors
\begin{equation}\label{logan}
D^n_{\underline{d},1^{g-2}}=-\lambda+\sum_{j=1}^n\begin{pmatrix}d_j+1\\2  \end{pmatrix}\psi_j
-0\cdot\delta_0-\sum_{\tiny{\begin{array}{cc}i,S\end{array}}}\begin{pmatrix}|d_S-i|+1\\2  \end{pmatrix}\delta_{i:S}
\end{equation}
in $\Pic_\QQ(\Mgnb)$, where $d_S:=\sum_{j\in S}d_j$.

\section{Divisor class computation}
In this section we compute the divisor classes defined above for the strata of quadratic differentials. Our strategy is to first compute the class of $Q^{2g-2}_{1^{2g-2},2^{g-1}}$ (which we will denote $Q_g$) in $\Pic_\QQ(\overline{\mathcal{M}}_{g,2g-2})$ via test curves. Utilising the symmetries in this class due to the symmetries in the signature we compute the intersection number of a number of test curves in $\overline{\mathcal{M}}_{g,2g-2}$ with $Q_g$ and the standard generators of $\Pic_\QQ(\overline{\mathcal{M}}_{g,2g-2})$ providing the class of $Q_g$. 

We then use gluing maps between moduli spaces of curves and other known classes from the strata of differentials to compute a formula for the general divisor $Q^{n}_{\underline{d},2^{g-1}}$ in $\Pic_\QQ(\overline{\mathcal{M}}_{g,2g-2})$ where $\underline{d}=(d_1,\dots,d_n)$ and $\sum d_i=2g-2$.

Consider the following test curves:
\begin{itemize}
\item
$A_{i:s}$: For $i=0,\dots,g$, $s=1,\dots,2g-2$ and $s\leq 2g-5$ if $i=g$ and $s\leq 2g-3$ if $i=g-1$. Fix general smooth pointed curves $[X,p_1,\dots,p_{s},x]\in\mathcal{M}_{i,s+1}$ and $[Y,p_{s+1},\dots,p_{2g-2}]\in \mathcal{M}_{g-i,2g-2-s}$. Form the curve by attaching $x$ to a point $y$ that varies in $Y$. See Figure 3.
\item
$B_{i:s}$: For $i=0,\dots,g$, $s=0,\dots,2g-3$ and $s\geq 2$ if $i=0$. Fix general smooth pointed curves $[X,p_1,\dots,p_{s},x]\in\mathcal{M}_{i,s+1}$ and $[Y,p_{s+2},\dots,p_{2g-2},y]\in \mathcal{M}_{g-i,2g-2-s}$. Form the curve by attaching $x$ to $y$ and allowing $p_{s+1}$ to vary freely in $X$. See Figure 3.
\item
$C_{i:s}$: For $i=0,\dots,g$, $s=0,\dots,2g-4$ and $s\geq 1$ if $i=0$ and $s\leq 2g-5$ for $i=g$. Fix general smooth pointed curves $[X,p_1,\dots,p_{s},x]\in\mathcal{M}_{i,s+1}$, $[Y,p_{s+3},\dots,p_{2g-2},y]\in \mathcal{M}_{g-i,2g-4-s}$ and $[Z,z_1,z_2,p_{s+2}]\in\mathcal{M}_{0,3}$. Form the curve by attaching $x$ to $z_1$, $y$ to $z_2$ and allowing $p_{s+1}$ to vary freely in $Z$. See Figure 3.
\end{itemize}
For some choices of parameters these curves coincide. Observe, for example, that $C_{0:1}=B_{0:2}$ and $A_{0:1}=\sigma B_{g,2g-3}$ where $\sigma$ permutes the first and $(2g-2)$th point.
\begin{center}
\vspace{0.5cm}
\begin{overpic}[width=1\textwidth]{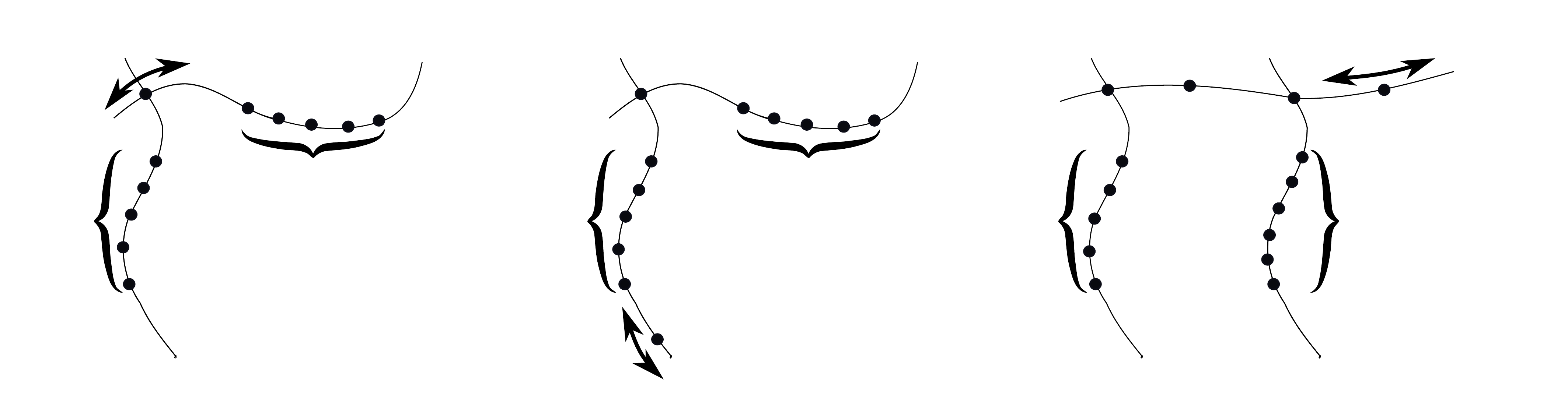}
\put (15, 0){$A_{i:s}$}
\put (47, 0){$B_{i:s}$}
\put (79, 0){$C_{i:s}$}
\footnotesize{
\put (-2.5, 12){$p_1,\dots,p_s$}
\put (-2, 9){$g(X)=i$}
\put (13, 15){{$p_{s+1},\dots,p_{2g-2}$}}
\put (14, 22){$g(Y)=g-i$}

\put (29.5, 12){$p_1,\dots,p_s$}
\put (30, 9){$g(X)=i$}
\put (45, 15){$p_{s+2},\dots,p_{2g-2}$}
\put (46, 22){$g(Y)=g-i$}
\put (43, 5){$p_{s+1}$}

\put (59.5, 12){$p_1,\dots,p_s$}
\put (60, 9){$g(X)=i$}
\put (86, 12){$p_{s+3},\dots,p_{2g-2}$}
\put (87, 9){$g(Y)=g-i$}
\put (74, 22.5){$p_{s+2}$}
\put (88, 18.5){$p_{s+1}$}
\put (92.5, 20){$g(Z)=0$}

}

\put (25,-6){\footnotesize{Figure 3: Test curves $A_{i:s}$, $B_{i:s}$ and $C_{i:s}$ in $\overline{\mathcal{M}}_{g,2g-2}$.}}
\end{overpic}
\end{center}
\vspace{01cm}

\begin{lem}\label{intbasis}
The following intersection numbers hold,
\begin{eqnarray*}\begin{array}{lcccccc}
A_{i:s}\cdot \delta_{i:\{1,\cdots,s\}}=-(4g-2i-4-s),\hspace{0.5cm} \\
A_{i:s}\cdot \delta_{i:\{1,\dots,s,j\}}=A_{i:s}\cdot \psi_j=1\text{ for $j=s+1,\dots ,2g-2,$}\\
\\
B_{i:s}\cdot \delta_{i:\{1,\dots,s\}}=1,\hspace{0.5cm} \\
B_{i:s}\cdot \delta_{i:\{1,\dots,s+1\}}=-1\hspace{0.5cm}\\
B_{i:s}\cdot \psi_{s+1}=2i-1+s,\\ 
B_{i:s}\cdot \psi_j= B_{i:s}\cdot \delta_{0:\{j,s+1\}}=1\text{ for $j=1,\dots ,s,$}\\
\\
C_{i:s}\cdot \delta_{i:\{1,\dots,s\}}=C_{i:s}\cdot \delta_{g-i:\{s+3,\dots,2g-2\}}=-1\\
C_{i:s}\cdot \psi_{s+1}=C_{i:s}\cdot \psi_{s+2}=C_{i:s}\cdot \delta_{0:\{s+1,s+2\}}=C_{i:s}\cdot \delta_{i:\{1,\dots,s+1\}}=C_{i:s}\cdot \delta_{g-i:\{s+3,\dots,2g-2,s+1\}}=1.
\end{array}
\end{eqnarray*}
All other intersections are zero.
\end{lem}
\begin{proof}
This is a simple exercise in intersection theory. See~\cite{HarrisMorrison}. Here we provide the example of test curve $A_{i:s}$. Clearly
\begin{equation*}
A_{i:s}\cdot \delta_{i:\{1,\dots,s,j\}}=A_{i:s}\cdot \psi_j=1\text{ for $j=s+1,\dots ,2g-2,$}
\end{equation*}
and the only other non-zero intersection is $A_{i:s}\cdot \delta_{i:\{1,\cdots,s\}}$ which we obtain by computing the degree of the restriction of the normal bundle of $\delta_{i:\{1,\cdots,s\}}$ to the curve $A_{i:s}$. Let $S$ be the surface obtained by blowing up $Y\times Y$ at the points $(p_i,p_i)$ for $i=s+1,\dots,2g-2$ and let $\pi:S\longrightarrow Y\times Y$ be the blowdown. If $\Delta$ denotes the diagonal in $Y\times Y$ and $\widetilde{\Delta}$ the proper transform in $S$ we have
\begin{equation*}
A_{i:s}\cdot \delta_{i:\{1,\cdots,s\}}=\deg(\mathcal{N}_{\widetilde{\Delta}/S}\otimes \mathcal{N}_{Y\times\{x\}/Y\times X})=\widetilde{\Delta}^2=-(4g-2i-4-s).
\end{equation*} 
\end{proof}

Intersecting these test curves directly with the divisor $Q^{2g-2}_{1^{2g-2},2^{g-1}}$, which we denote $Q_g$, we obtain the following.
\begin{lem}\label{intQg1}
The following intersection numbers hold,
\begin{eqnarray*}\begin{array}{llllll}
A_{i:s}\cdot Q_g= \begin{cases} 4^{g-1}(s-2i)^2(g-i)    &\text{ for $s\ne 0,2g-2$}\\
(4^{g-i}-1)4^i(g-i-1)^2(g-i)+4^i(g-i)(g-i+1)(g-i-1)      &\text{ for $s=2g-2.$}\end{cases} 
\end{array}
\end{eqnarray*}
\end{lem}

\begin{proof}
Observe $A_{i:s}$ for $s\ne 0,2g-2$ intersects $Q_g$ for $y\in Y$ such that there exists $\{q_j\}$ with 
\begin{eqnarray*}
(\omega_X)=(2i-s-4)x+\sum_{j=1}^s p_j+2\sum_{j=1}^i q_j\sim 2K_X\\
(\omega_Y)=(s-2i)y+\sum_{j=s+1}^{2g-2}p_j+2\sum_{j=i+1}^{g-1}q_j\sim 2K_Y.
\end{eqnarray*}
By the Picard variety method~[\cite{Mullane1},\S2.5] for any general genus $g$ curve $C$ and non-zero integers $d_i$, the degree of the map
\begin{eqnarray*}
\begin{array}{cccc}
C^g&\rightarrow&\Pic^{\sum d_i}(C)\\
(p_1,\dots,p_g)&\mapsto&\OO_X(\sum_{i=1}^g d_i p_i)
\end{array}
\end{eqnarray*} 
is $\prod_{i=1}^g d_i^2 g!$. Hence the degree of the map
\begin{eqnarray*}
\begin{array}{cccc}
X^i&\rightarrow&\Pic^{2i}(X)\\
(q_1,\dots,q_i)&\mapsto&\OO_X(2q_1+\dots+2q_i)
\end{array}
\end{eqnarray*} 
is $4^i i!$ while the degree of the map 
\begin{eqnarray*}
\begin{array}{cccc}
Y^{g-i}&\rightarrow&\Pic^{2g-4i+s-2}(Y)\\
(y,q_{i+1},\dots,q_{g-1})&\mapsto&\OO_Y((s-2i)y+2q_{i+1}+\dots+2q_{g-1})
\end{array}
\end{eqnarray*} 
is $4^{g-i-1}(s-2i)^2(g-i)!$ and accounting for the ordering of the $q_j$ we obtain
\begin{equation*}
A_{i:s}\cdot Q_g=4^{g-1}(s-2i)^2(g-i)
\end{equation*}
for $s\ne 0,2g-2$. When $s=2g-2$ the situation is slightly more complicated. In this case we have two possibilities. In the first case we require $y$ and $\{q_j\}$ such that 
\begin{eqnarray*}
(\omega_X)=(2(i-g)-2)x+\sum_{j=1}^{2g-2} p_j+2\sum_{j=1}^i q_j\sim 2K_X\\
(\omega_Y)=(2(g-i)-2)y+2\sum_{j=i+1}^{g-1}q_j\sim 2K_Y.
\end{eqnarray*}
and further the global $k$-residue condition requires
\begin{equation*}
((g-i)-1)y+\sum_{j=i+1}^{g-1}q_j\nsim K_Y.
\end{equation*}
and hence by enumerating the solutions to 
\begin{equation*}
((g-i)-1)y+\sum_{j=i+1}^{g-1}q_j\sim \eta_Y\otimes K_Y.
\end{equation*}
for the $4^{g-i}-1$ bundles $\eta_Y\nsim\OO_Y$ with $\eta_Y^{\otimes 2}\in \OO_Y$ we obtain
\begin{equation*}
(4^{g-i}-1)4^i(g-i-1)^2(g-i)
\end{equation*}
smoothable twisted $k$-canonical divisors of this type. The second case is $y$ and $\{q_j\}$ such that 
\begin{eqnarray*}
(\omega_X)=(2(i-g)-2)x+\sum_{j=1}^{2g-2} p_j+2\sum_{j=1}^{i+1} q_j\sim 2K_X\\
(\omega_Y)=(g-i)y+\sum_{j=i+2}^{g-1}q_j\sim K_Y.
\end{eqnarray*}
and further the global $k$-residue condition requires
\begin{equation*}
\res^2_{x}\omega_X=0.
\end{equation*}
 As $x$ and $p_j$ are chosen general in $X$, we have that for each of the $4^i$ bundles $\eta_X$ such that 
 \begin{equation*}
 \eta_{X}^{\otimes 2}\sim 2K_X(-(2(i-g)-2)x-\sum_{j=1}^{2g-2} p_j)
 \end{equation*}
there is a unique $\omega_{X}$ (up to the scaling by $\CC^*$) that satisfies this. Two such differentials would provide a one dimensional family of $q_j$ on $X$ satisfying the above and hence contradicting the points $x$ and $p_j$ being in general position. In $Y$ the condition is that $y$ is a Weierstrass point. There are $(g-i)(g-i-1)(g-i+1)$ such points and hence 
\begin{equation*}
4^i(g-i)(g-i+1)(g-i-1)
\end{equation*}
solutions of this type.
\end{proof}

\begin{lem}\label{intQg2}
The following intersection numbers hold,
\begin{eqnarray*}\begin{array}{llllll}
B_{i:s}\cdot Q_g= \begin{cases} 4^{g-1}i &\text{ for $s\ne 0,2g-3$}\\
 4^{g-1}i-4^{g-i}i &\text{ for $s=0$.} \\
  \end{cases} 
\end{array}
\end{eqnarray*}
\end{lem}

\begin{proof}
We require points $p_{s+1}$ and $\{q_j\}$ such that
\begin{eqnarray*}
(\omega_X)=(2i-s-3)x+p_{s+1} +\sum_{j=1}^s p_j+2\sum_{j=1}^{i-1} q_j\sim 2K_X\\
(\omega_Y)=(s-2i-1)y+\sum_{j=s+2}^{2g-2}p_j+2\sum_{j=i}^{g-1}q_j\sim 2K_Y.
\end{eqnarray*}
Using the Picard variety method~[\cite{Mullane1},\S2.5] as above to obtain the degree of the relevant maps we obtain $4^{i-1}i!$ solutions on $X$ and $4^{g-i}(g-i)!$ solutions on $Y$. Hence allowing for labelling we obtain $B_{i:s}\cdot Q_g=4^{g-1}i$ for $s\ne 0,2g-3$. When $s=0$ we must omit the unique solution in $X$ where $x=p_{s+1}$ which by~[\cite{Mullane1},\S2.5] occurs with multiplicity $i$. Hence $B_{i:0}\cdot Q_g=4^{g-1}i-4^{g-i}i$.
\end{proof}

\begin{lem}\label{intQg3}
The following intersection numbers hold,
\begin{eqnarray*}\begin{array}{llllll}
C_{i:s}\cdot Q_g= \begin{cases} 0 &\text{ for $s\ne 0,2g-4$}\\
 4^{g-i}i &\text{ for $s=0$} \\
  4^{i}(g-i)&\text{ for $s=2g-4$.} \end{cases}
\end{array}
\end{eqnarray*}
\end{lem}

\begin{proof}
For $s\ne 0, 2g-4$ a twisted quadratic differential of the type required would violate the assumption that the chosen pointed curves are general, hence $C_{i:s}\cdot Q_g=0$ in these cases.

Consider $s=0$. Smoothable quadratic differentials of the type required 
\begin{eqnarray*}
(\omega_X)=(2i-2)x+2\sum_{j=1}^{i-1} q_j\sim 2K_X\\
(\omega_Y)=-2iy+\sum_{j=3}^{2g-2}p_j+2\sum_{j=i}^{g-1}q_j\sim 2K_Y\\
(\omega_Z)=  (-2i-2)z_1    +  (2i-4)z_2  + p_{1}+p_{2}\sim 2K_Z
\end{eqnarray*}
Using the Picard variety method~[\cite{Mullane1},\S2.5] as above there are $4^{g-i}(g-i)!$ such solutions on $Y$. But the only solution on $X$ comes from a square of the unique such holomorphic differential. Hence the global $k$-residue condition requires the $\res^2_Z(\omega_Z)=0$. But as $Z$ is a rational curve, using the cross-ratio sending $z_1$, $p_2$ and $z_2$ to $0,1$ and $\infty$ and letting $p_1=b$ we obtain
\begin{equation*}
\omega_Z=\frac{(z-1)(z-b)}{z^{2i+2}}(dz)^2.
\end{equation*}
Taking the square root we obtain the residue as a degree $i$ polynomial in $b$ giving $i$ solutions for $b$. Hence $C_{i:0}\cdot Q_g=4^{g-i}i$. The case $s=2g-2$ follows by symmetry.
\end{proof}

Comparing the intersection of curves $A_{i:s}$, $B_{i:s}$ and $C_{i:s}$ with the standard generators of $\Pic_\QQ(\Mbar{g}{2g-2})$ (Lemma~\ref{intbasis}) and $Q_g$ (Lemmas~\ref{intQg1}, \ref{intQg2}, \ref{intQg3}) we obtain the following.

\Qg*

\begin{proof}
Observe that $c_{i:S}=c_{i:T}$ for any $S,T\subset \{1,\dots,2g-2\}$ such that $|S|=|T|$ and similarly $c_{\psi_j}=c_{\psi_i}$ for $i,j\in \{1,\dots,2g-2\}$. Hence we let $c_{i:s}=c_{i:S}=c_{i:T}$ where $s=|S|=|T|$ and $c_{\psi}=c_{\psi_j}=c_{\psi_i}$. With these identifications and Lemma~\ref{intbasis} we obtain the following. 
\begin{eqnarray*}\begin{array}{lll}
A_{i:s}\cdot Q_g=(2g-2-s)(c_\psi+c_{i:s+1})-(4g-2i-4-s)c_{i:s}\\
B_{i:s}\cdot Q_g=(2i+2s-1)c_\psi+sc_{0:2}+c_{i:s}-c_{i:s+1}  \\
C_{i:s}\cdot Q_g= 2c_\psi+c_{0:2}+c_{g-i:2g-s-3}-c_{g-i:2g-s-4}+c_{i:s+1}-c_{i:s}.
\end{array}
\end{eqnarray*}
where by convention $c_{0:1}=-c_\psi$. Hence by Lemmas~\ref{intQg1}, \ref{intQg2} and \ref{intQg3} we obtain all coefficients except $c_0$ and $c_\lambda$. Observe that the curves $A_{i:s}$ alone are enough to compute these coefficients, while curves $B_{i:s}$ and $C_{i:s}$ provide cross-checks.

Fix a general smooth pointed curve $[Y,y,q_1,\dots,q_{2g-2}]\in\mathcal{M}_{g-2,2g-1}$ and consider the map
\begin{eqnarray*}
\begin{array}{cccc}
\pi:&\overline{\mathcal{M}}_{2,1}&\rightarrow& \overline{\mathcal{M}}_{g,2g-2}\\
&(X,x)&\mapsto&(X\bigcup_{x=y}Y,q_1,\dots,q_{2g-2}).
\end{array}
\end{eqnarray*} 
Set theoretically, $\pi^*Q_g$ has two components. One where $x$ is a Weierstrass point, the other where $x$ is zero of a quadratic differential of signature $(2,2)$. However, these two conditions are known to coincide. Observe
\begin{equation*}
\pi^*Q_g=-c_{2:\emptyset}\psi+(c_\lambda+c_0)\lambda+(c_{1:\emptyset}-2c_0)\delta_{1:{1}}
\end{equation*}
via the known pullback of the generators under $\pi$ and the extra relation $\lambda=\frac{1}{10}\delta_0+\frac{1}{5}\delta_{1:\{1\}}$ in $\Pic_\QQ(\overline{\mathcal{M}}_{2,1})$. Hence as $c_{2:\emptyset}=-18(4^{g-2})$ we have $\pi^*Q_g=6(4^{g-2})W$ where 
$W=3\psi-\lambda-\delta_{1,\{1\}}$
is the known class of the Weierstrass divisor~\cite{Cukierman}. This gives the remaining coefficients.
\end{proof}


Pulling back classes under gluing morphisms allows us to leverage our understanding of the degeneration of quadratic differentials by the class of $Q_g$ and known classes in the $k=1$ case to obtain a general formula for the divisors of interest.

\Qd*

\begin{proof}
Fix a general smooth pointed curve $[Y,y,q_1,\dots,q_{s}]\in\mathcal{M}_{i,s+1}$ and consider the map
\begin{eqnarray*}
\begin{array}{cccc}
\pi:&\overline{\mathcal{M}}_{g-i,n+1-s}&\rightarrow& \overline{\mathcal{M}}_{g,n}\\
&(X,x,q_{s+1},\dots,q_{n})&\mapsto&(X\bigcup_{x=y}Y,q_1,\dots,q_{n}).
\end{array}
\end{eqnarray*} 
For any $S\subset\{1,\dots,n\}$ with $|S|=s$, after possibly reordering $\underline{d}$ without loss of generality, consider $S=\{1,\dots,s\}$, define
\begin{eqnarray*}
\underline{d}'&=& (\sum_{j=1}^s d_j-2i, d_{s+1},\dots, d_n)  \\
\underline{d}''&=& (1-i+\frac{1}{2}\sum_{j=1}^s d_j, \frac{d_{s+1}}{2},\dots, \frac{d_n}{2})
\end{eqnarray*}
and specify three cases
\begin{eqnarray*}\begin{array}{ll}
\text{Case A:} & \text{$\underline{d}$ contains only even non-negative entries and $\underline{d}''$ contains only non-negative integer entries,}\\
\text{Case B:} & \text{$\underline{d}$ contains an odd or negative entry and $\underline{d}''$ contains only non-negative integer entries,}\\
\text{Case C:} & \text{$d_j$ is odd or negative for some $j\in S$ or $\sum_{j\in S}d_j\geq 2i$ and Case B is not satisfied.}
\end{array}\end{eqnarray*}
Then 
\begin{equation*}
\pi^*Q^n_{\underline{d},2^{g-1}}=\begin{cases} 4^iQ^{n+1-s}_{\underline{d}',2^{g-i-1}}+(4^i-1)D^{n+1-s}_{\underline{d}'',1^{g-i-2}} +\text{ Boundary classes,} &\text{for Case A}\\
4^iQ^{n+1-s}_{\underline{d}',2^{g-i-1}}+4^iD^{n+1-s}_{\underline{d}'',1^{g-i-2}} +\text{ Boundary classes,}&\text{for Case B}\\
 4^iQ^{n+1-s}_{\underline{d}',2^{g-i-1}} +\text{ Boundary classes,}&\text{for Case C.}
\end{cases}
\end{equation*}
The cases occur when the choice of $S$ splits the signature such that a new component arises coming from quadratic differentials that are the square of holomorphic differentials. The components are clear set theoretically from applying the global $k$-residue condition. Note that the global $k$-residue condition also implies that the boundary components in the equations do not contain $\delta_{1:\emptyset}$ or $\delta_{0:T}$ for any $T$. Setting $i=0$ the multiplicities in the equations are clear and we obtain the coefficients of $\lambda, \psi_j, \delta_0$ and $\delta_{0:S}$. These values then give the multiplicities in the equations for $i\geq 1$. The relations in the three cases described are enough to complete the class computation.
\end{proof}

\section{Rigidity and extremality of divisors}

An effective divisor $D$ is \emph{extremal} or spans an extremal ray in the pseudo-effective cone if the divisor class $D$ cannot be written as a sum $m_1D_1+m_2D_2$ of pseudo-effective $D_i$ with $m_1,m_2>0$ with $D_i$ non-proportional classes. An effective divisor $D$ is \emph{rigid} if $h^0(mD)=\dim H^0(mD)=1$ for every positive integer $m$. 

A curve $B$ contained in an effective divisor $D$ is known as a \emph{covering curve} for $D$ if irreducible curves with numerical class equal to $B$ cover a Zariski  dense subset of $D$. The following lemma provides a well-known trick for showing a divisor is rigid and extremal.
\begin{lem}\label{lemma:covering}
If $B$ is a covering curve for an irreducible effective divisor $D$ with $B\cdot D<0$ then $D$ is rigid and extremal.
\end{lem}

\begin{proof}
~\cite[Lemma 4.1]{ChenCoskun}.  
\end{proof}

A curve $B$ is known as a \emph{moving curve} if $B\cdot D\geq 0$ for all effective divisors $D$. Fibrations offer one method of obtaining moving curves as if the numerical equivalence classes of $B$ cover a Zariski dense subset of $X$ and the general curve is irreducible then necessarily, $B\cdot D\geq 0$ for all effective divisors $D$. Consider the curve $B^{n,k}_{\mu}$ defined as
\begin{equation*}
B^{n,k}_{\mu}:=\overline{\bigl\{ [C,p_1,...,p_n]\in\mathcal{M}_{g,n} \hspace{0.15cm}\big| \hspace{0.15cm} \text{fixed general $[C,p_{g+2},...,p_m]\in\mathcal{M}_{g,m-g-1}$ and $[C,p_{1},...,p_m]\in\PPP^k(\mu)$}    \bigr\}   }.
\end{equation*}

\begin{prop}\label{moving}
The class of $B^{n,k}_\mu$ is a moving curve in $\Mgnb$ when $|\mu|\geq n+g$.
\end{prop}

\begin{proof}
Fix general $[C,p_1,...,p_n]\in\mathcal{M}_{g,n}$. To be contained in a numerical equivalence class of a curve $B^{n,k}_\mu$ we require $p_{n+1},...,p_m$ on $C$ such that 
\begin{equation*}
\sum_{i=1}^mk_ip_i\sim kK_C.
\end{equation*}
Let $d=\sum_{i=1}^nd_i$ and consider the map
\begin{eqnarray*}
\begin{array}{cccccc}
f:C^{m-n}&\longrightarrow &\Pic^{d}(C)\\
(p_{n+1},...,p_m)&\longmapsto&kK_C(-\sum_{i=n+1}^md_ip_i).
\end{array}
\end{eqnarray*}
The domain and target have dimension $m-n$ and $g$ respectively. Hence curves with numerical equivalence class equal to $B^{n,k}_\mu$ for $|\mu|\geq n+g$ cover a Zariski dense subset of $\Mgnb$ as the fibre $f^{-1}(\sum_{i=1}^nd_ip_i)$ will be non-empty for general $[C,p_1,...,p_n]$ when $m\geq n+g$. 

There remains the possibility that a general curve with numerical class $B^{n,k}_{\mu}$ is reducible with components of differing class. However, if this were to occur, taking the closure of a distinguishable component over all $[C,p_{n+1},...,p_m]\in\mathcal{M}_{g,m-n}$ contradicts the irreducibility of  $\Mgn$. Hence if $B^{n,k}_{\mu}$ is reducible, all components of the general curve must have the same class and are hence proportional to $B^{n,k}_{\mu}$.
\end{proof}

\begin{restatable}{thm}{boundary}
\label{thm:boundary}
For $g\geq 2$, $k$-signature $\mu$ with $|\mu|=n+g-1$ for $n\geq g+1$,
\begin{equation*}
B^{n,k}_{\mu,1,-1}\cdot D^{n,k}_\mu=0.
\end{equation*}
\end{restatable}

\begin{proof}
Fix general $[C,p_{g+2},\dots,p_{n+g+1}]\in \Mgn$. If $B^{n,k}_{\mu,1,-1}$ and $D^{n,k}_{\mu}$ have non-empty intersection we require some $q_1,...,q_{g-1}\in C$ such that
\begin{equation*}
\sum_{i=1}^{n+g-1}k_ip_i   +p_{g+n}-p_{g+n+1}\sim \sum_{i=1}^{n}k_ip_i+\sum_{i=1}^{g-1}k_{i+n}q_i\sim kK_C
\end{equation*}
and hence 
\begin{equation*}
\sum_{i=n+1}^{n+g-1}k_ip_i+p_{g+n}-p_{g+n+1}-\sum_{i=1}^{g-1}k_{i+n}q_i\sim \OO_C.
\end{equation*}
This implies the existence of a degree $d=1+\sum|k_i| $ cover of $\PP^1$ with ramification profile above $0$ and $\infty$ given by the points with positive and negative coefficients respectively in the equation above. Riemann-Hurwitz implies that in a general such cover there will be $4g-2$ other simple ramification points. Hence the dimension of the space of such covers is $4g-3$ and as $g+1$ points are fixed general and $\dim(\mathcal{M}_{g,g+1})=4g-2$ there does not exist such $q_i$ for a general choice of $p_i$ for $i=n+1,...,g+n+1$.
\end{proof}

\begin{prop}\label{prop:covering}
If $D^{n,k}_\mu$ is an irreducible divisor then $B^{n,k}_\mu$ is the class of a covering curve. 
\end{prop}

\begin{proof}
The numerical classes of $B^{n,k}_{\mu}$ cover a Zariski dense subset of $D^{n,k}_\mu$. If the general curve $B^{n,k}_\mu$ is reducible the components must have the same class as otherwise, taking the closure of one of these components over all numerical classes of $B^{n,k}_\mu$ would contradict the irreducibility of $D^{n,k}_\mu$. Hence irreducible curves with class proportional to $B^{n,k}_\mu$ cover a Zariski dense subset of $D^{n,k}_\mu$.
\end{proof}

By observing that for some signatures $\mu$ the curve $B^{n,k}_{\mu}$  forms a component of a specialisation of $B^{n,k}_{\mu,-1,1}$ we obtain the negative intersection of the covering curve for the corresponding divisors $D^{n,k}_\mu$.

\begin{prop}\label{prop:intersection}
$B^{g+1,k}_{\underline{d},k^{g-1}}\cdot D^{g+1,k}_{\underline{d},k^{g-1}}<0$ 
for $g\geq 2$ and $\underline{d}=(d_1,d_2,d_3,k^{g-2})$ with $d_1+d_2+d_3=k$ and $k\nmid d_i$ for some $i$.
\end{prop}

\begin{proof}
Our proof generalises the proof of Proposition 5.4 in~\cite{Mullane3} via the global $k$-residue condition. Theorem~\ref{thm:boundary} implies
\begin{equation*}
B^{g+1,k}_{\underline{d},k^{g-1},1,-1}\cdot D^{g+1,k}_{\underline{d},k^{g-1}}=0
\end{equation*}
Curves numerically equivalent to $B^{g+1,k}_{\underline{d},k^{g-1},1,-1}$ are constructed as one-dimensional fibres of the morphism forgetting the first $g+1$ points
$$\overline{\PPP}^k(\underline{d},k^{g-1},1,-1)\longrightarrow\overline{\mathcal{M}}_{g,g+1}.$$
Consider the fibre above the nodal pointed curve with $p_{g+2},...,p_{2g}$ general points in $C$, a general genus $g$ curve and $p_{2g+1}$ and $p_{2g+2}$ lying on a rational tail $X$ above a general point $x$ on the curve $C$. We describe the fibre by applying the global $k$-residue condition to find the position of $p_1,\dots,p_{g+1}$ that give smoothable twisted $k$-canonical divisors of the required type.

\begin{figure}[htbp]
\begin{center}
\begin{overpic}[width=1\textwidth]{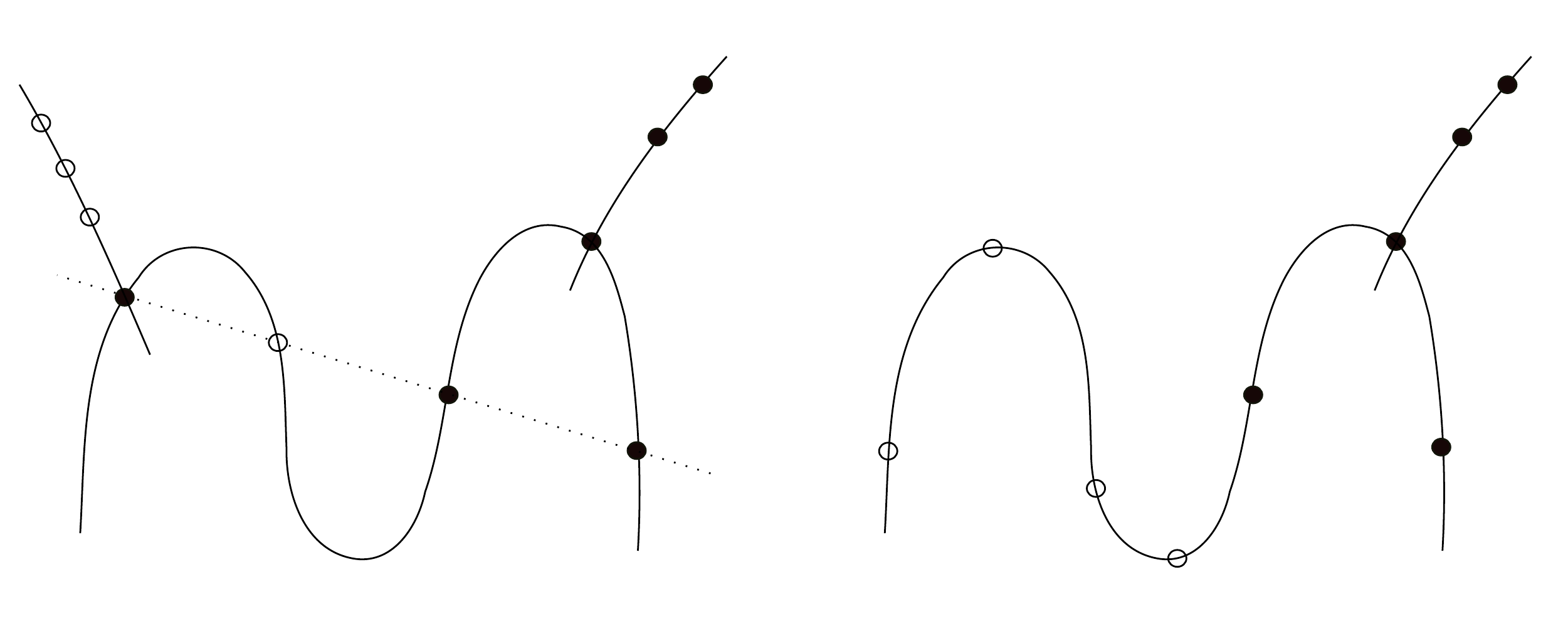}
\put (2,8){$C$}
\put (10,30){$Y\cong \PP^1$}
\put (4,33){$p_1$}
\put (5.5,30){$p_2$}
\put (7,26.5){$p_3$}
\put (5,20.5){$y$}
\put (15,17){$p_4$}
\put (25.5,14){$p_5$}
\put (38,10){$p_6$}
\put (41,27){$X\cong \PP^1$}
\put (39,32){$p_7$}
\put (41.5,35){$p_8$}
\put (37,27){$x$}
\put (13,0){$y+p_4+p_5+p_6\sim K_C$}

\put (53,8){$C$}
\put (58,11.5){$p_1$}
\put (63,27){$p_2$}
\put (66.5,9){$p_3$}
\put (74.5,7){$p_4$}
\put (76.5,14){$p_5$}
\put (89,10){$p_6$}
\put (92,27){$X\cong \PP^1$}
\put (90,32){$p_7$}
\put (92.5,35){$p_8$}
\put (88,27){$x$}
\put (55,0){$d_1p_1+d_2p_2+d_3p_3+kp_4+kp_5+kp_6\sim kK_C$}

\put (13,-7){\footnotesize{Figure 4: Two components of $B^{4,k}_{d_1,d_2,d_3,k^3,1,-1}$ in $\overline{\mathcal{M}}_{3,4}$ when $p_7$ and $p_8$ sit on a $\PP^1$ tail.}}
\end{overpic}
\vspace{0.6cm}
\end{center}
\end{figure}

Fix $[C\cup_{x=x'}X,p_{g+2},\dots,p_{2g+2}]\in \overline{\mathcal{M}}_{g,g+1}$ as described above and consider the associated curve $B^{g+1,k}_{\underline{d},k^{g-1},1,-1}$. The smoothable twisted $k$-canonical divisors of interest are of two different types. Denote by $\delta(B^{g+1,k}_{\underline{d},k^{g-1},1,-1})$ the curve contained in the boundary of $\overline{\mathcal{M}}_{g,g+1}$. Consider $d_i\ne k$ for $i=1,2,3$. In this case $\delta(B^{g+1,k}_{\underline{d},k^{g-1},1,-1})$  is the curve created by the points $p_1,p_2,p_3$ moving freely on a $\PP^1$-tail we denote $Y$ attached to the curve $C$ at a point $y$. The resulting twisted $k$-canonical divisor is of the form
\begin{eqnarray*}
y+\sum_{i=4}^{2g}p_i &\sim& K_C\\
p_{2g+1}-p_{2g+2}-2kx&\sim& kK_X\\
d_1p_1+d_2p_2+d_3p_3-3ky&\sim& kK_Y.
\end{eqnarray*}
The pointed nodal curve in the left of Figure 4 shows this situation for $g=3$ which is presented in detail in Example~\ref{Ex2}. On a general curve $C$ with $p_{g+2},\dots,p_{2g}$ fixed we have $h^0(K_C-p_{g+2}-...-p_{2g})=1$ providing $(g-1)!$ solutions for $y,p_4,...,p_{g+1}$. It remains to check the global $k$-residue condition to find which such twisted $k$-canonical divisors are smoothable. Observe for any $k$-differential $\omega_X$ on $X$ with $(\omega_X)\sim p_{2g+1}-p_{2g+2}-2kx$ we have $\res^k_x(\omega_X)\ne 0$. Hence as discussed in Example~\ref{Ex2}, Figure 5 gives the two possible level graphs to provide smoothable twisted $k$-canonical divisors of this type. The global $k$-residue condition from Graph A is
\begin{equation*}
\res^k_x(\omega_X)+(-1)^k\res^k_y(\omega_Y)=0,
\end{equation*}
which can be obtained for any configuration of points on rational curve $Y$ with $\res_{y}^k(\omega_Y)\ne 0$ by scaling the differential. Graph B gives the condition
\begin{equation*}
\res^k_y(\omega_Y)=0.
\end{equation*}
Hence all configurations of $y,p_1,p_2,p_3$ on the rational tail $Y$ are included. 

In the case that some $d_i$, say $d_3=k$, the curve  $\delta(B^{g+1,k}_{\underline{d},k^{g-1},1,-1})$ will include more components. There will be $g-2$ components specified by $i=3,...,g+1$ where  $p_1,p_2$ and $p_i$ move freely on a $\PP^1$-tail attached to $C$ at a point $y$ such that $y+p_3+...+p_{i-1}+p_{i+1}+...+p_{2g}\sim K_C$. Further, there is a component where $p_3,...p_{g+1}$ are fixed such that $\sum_{i=3}^{2g}p_i\sim K_C$ and $p_1$ and $p_2$ sit on a rational tail attached to $C$ at a point $y$ that varies freely in $C$. 
\begin{figure}[htbp]
\begin{center}
\vspace{1cm}
\begin{overpic}[width=0.7\textwidth]{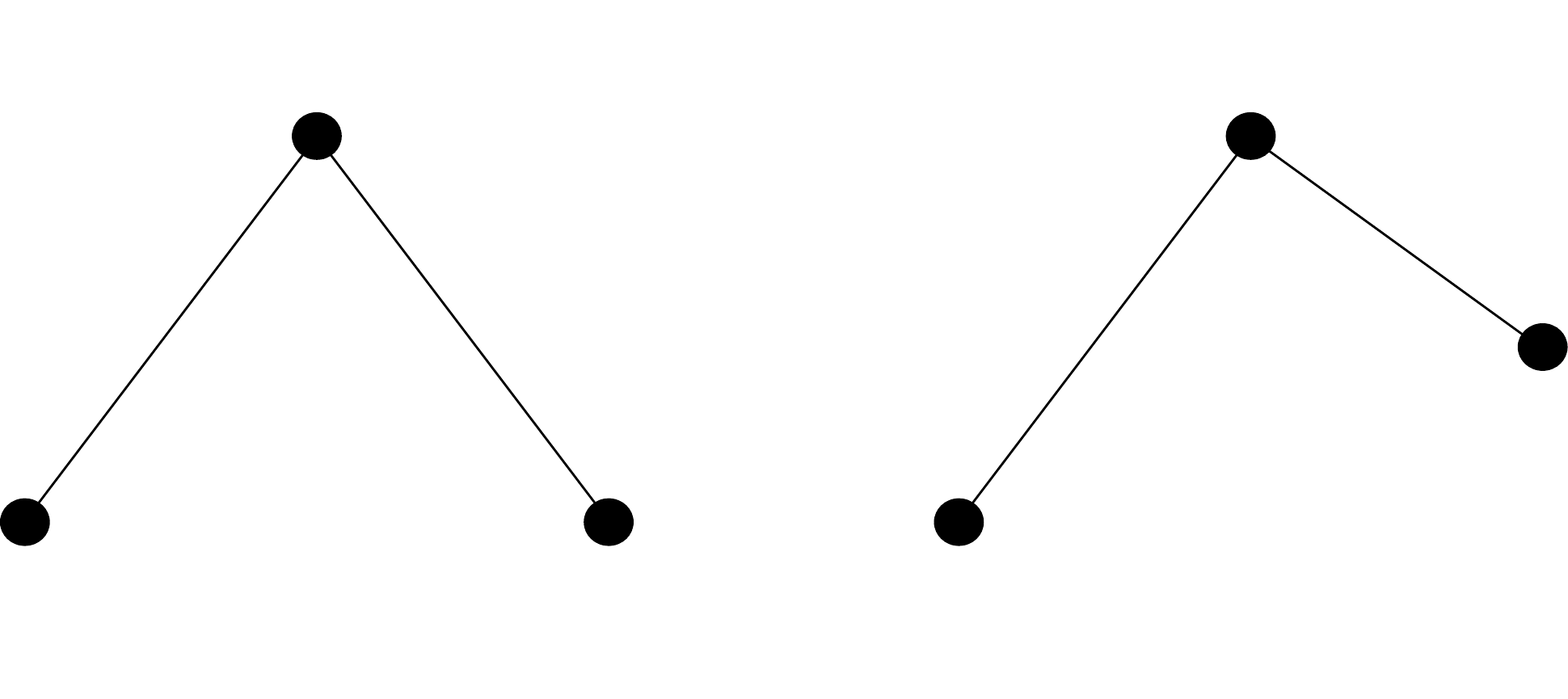}

\put (19,40){$C$}
\put (0,6){$X$}
\put (38,6){$Y$}
\put (13,47){Graph A}

\put (78.5,40){$C$}
\put (59,6){$X$}
\put (97,17){$Y$}
\put (72,47){Graph B}

\put (11,-0){\footnotesize{Figure 8: Level graphs giving the global $k$-residue condition.}}
\end{overpic}
\end{center}
\end{figure}

The other candidate for smoothable twisted $k$-canonical divisors in this special fibre are twisted $k$-canonical divisors of the form
\begin{eqnarray*}
d_1p_1+d_2p_2+d_3p_3+\sum_{i=4}^{2g}kp_i &\sim& kK_C\\
p_{2g+1}-p_{2g+2}-2kx&\sim& kK_X,
\end{eqnarray*}
where the points $p_1,...,p_{g+1}$ vary in a one dimensional family in the curve $C$. The pointed curve in the right of Figure 4  shows this situation for $g=3$. As there are poles on both components the global $k$-residue condition is empty and all solutions of this type are smoothable. Hence this component is simply the curve $B^{g+1,k}_{{\underline{d},k^{g-1}}}$.

Hence
\begin{equation*}
B^{g+1,k}_{\underline{d},k^{g-1},1,-1}\sim \delta(B^{g+1,k}_{\underline{d},k^{g-1},1,-1})+B^{g+1,k}_{\underline{d},k^{g-1}}.
\end{equation*}
As detailed in Example~\ref{Ex1},  $\delta(B^{g+1,k}_{\underline{d},k^{g-1},1,-1})$ intersects $D^{g+1,k}_{\underline{d},k^{g-1}}$ at the non-empty finite set of points where the $k$-residue at $y$ is zero. Hence
\begin{equation*}
 \delta(B^{g+1,k}_{\underline{d},k^{g-1},1,-1})\cdot D^{g+1,k}_{\underline{d},k^{g-1}}>0,
\end{equation*} 
but by Theorem~\ref{thm:boundary}
\begin{equation*}
B^{n,k}_{\mu,1,-1}\cdot D^{n,k}_\mu=0
\end{equation*}
and hence
\begin{equation*}
 B^{g+1,k}_{\underline{d},k^{g-1}}\cdot D^{g+1,k}_{\underline{d},k^{g-1}}<0.
\end{equation*} 
\end{proof}

\extremal*

\begin{proof}
Proposition~\ref{prop:covering} and Proposition~\ref{prop:intersection} give a covering curve for the divisors with negative intersection. Hence by Lemma~\ref{lemma:covering} the divisors are rigid and extremal.
\end{proof}

\begin{rem}
Pullbacks of the above rigid and extremal divisors under forgetful morphisms are also rigid and extremal. Fix $n>g+1$ and let $\varphi:\Mgnb\longrightarrow \overline{\mathcal{M}}_{g,g+1}$ forget all but the first $g+1$ points we have $\varphi^*D^{g+1,k}_{\underline{d},k^{g-1}}$ is also rigid and extremal. Observe that $\varphi^*D^{g+1,k}_{\underline{d},k^{g-1}}=D^{n,k}_{\underline{d},0^{n-g-1},k^{g-1}}$ is irreducible as the pullback of an irreducible divisor intersecting the interior of $\Mgnb$. Further, $B^{n,k}_{\underline{d},0^{n-g-1},k^{g-1}}$ provides a covering curve for $\varphi^*D^{g+1,k}_{\underline{d},k^{g-1}}$ with 
\begin{equation*}
\varphi_*B^{n,k}_{\underline{d},0^{n-g-1},k^{g-1}}=B^{g+1,k}_{\underline{d},k^{g-1}}.
\end{equation*}
Hence by the projection formula
\begin{equation*}
B^{n,k}_{\underline{d},0^{n-g-1},k^{g-1}}\cdot \varphi^*D^{g+1,k}_{\underline{d},k^{g-1}}=\varphi_*B^{n,k}_{\underline{d},0^{n-g-1},k^{g-1}}\cdot D^{g+1,k}_{\underline{d},k^{g-1}}=B^{g+1,k}_{\underline{d},k^{g-1}}\cdot D^{g+1,k}_{\underline{d},k^{g-1}}<0.
\end{equation*}
\end{rem}

Though the classification of the connected components of the meromorphic $k$-strata of differentials is still in progress, Chen and Gendron~\cite{ChenGendron} have completed the $k=2$ case of quadratic differentials for non-simple poles, which with Lanneau's classification of the finite area cases~\cite{Lanneau} (See Corollary~\ref{irredQ}) provides the following result.  

\quaddiv*

\section{Rigidity and extremality of higher codimension cycles}

In this section we use the property of rigidity inductively on the codimension under the forgetful morphism that forgets marked points. The general strategy, first used by Chen and Tarasca~\cite{ChenTarasca} to show that marking Weierstrass points on genus $g=2$ curves gave rigid and extremal cycles in $\Mbar{2}{n}$ for $n=2,\dots,6$, is to assume there exists a non-trivial effective decomposition. By pushing this decomposition forward under the forgetful morphism we use the assumed rigidity in lower codimension to deduce that the assumed decomposition must include the original cycle hence contradicting the assumed non-triviality. Some complications arise in low genus $g=2,3$ when there are multiple candidates for the cycle we find in the decomposition. We deal with these cases by ruling out the extraneous candidates by showing their inclusion would violate our effectivity assumption through the use of moving curves constructed in Proposition~\ref{moving}.

\fullextremal*

\begin{proof}
The cases for $k=1$ appear in~\cite{Mullane3} in the case of divisors, that is, $j=g-1$ and \cite{Mullane4} for $j=0,\dots,g-2$. The divisorial case for $k\geq 2$ is Theorem~\ref{thm:extremal}. The remaining cases follow below as the content of Propositions \ref{mero1}, \ref{mero2} and \ref{mero3}. 
\end{proof}

Again, while the classification of the connected components of the strata of $k$-differentials is still in progress, the completed case of quadratic differentials given in Corollary~\ref{irredQ} yields the following.

\Quadfullextremal

We begin by proving the most general case where the complications mentioned above do not occur.

\begin{prop}\label{mero1}
For $g\geq 3$, $k\geq 2$ and $j=0,\dots,g-1$ if $\PPP^k(d_1,d_2,d_3,k^{2g-3})$ is irreducible the cycle $[\varphi_j{}_*\overline{\PPP}^k(d_1,d_2,d_3,k^{2g-3})]$ is rigid and extremal in $\mbox{Eff}^{\hspace{0.1cm}g-j}(\overline{\mathcal{M}}_{g,2g-j})$, where $\varphi_j:\overline{\mathcal{M}}_{g,2g}\longrightarrow\overline{\mathcal{M}}_{g,2g-j}$ forgets the last $j$ points with $d_1+d_2+d_3=k$ and some $d_i=k$ if $g=3$.
\end{prop}

\begin{proof}
Proceed by induction. Assume $[(\varphi_{j+1})_*\overline{\PPP}^k(d_1,d_2,d_3,k^{2g-3})]$ is rigid and extremal. 
If $[(\varphi_j)_*\overline{\PPP}^k(d_1,d_2,d_3,k^{2g-3})]$ is not extremal then it can be expressed as
\begin{equation*}
[(\varphi_j)_*\overline{\PPP}^k(d_1,d_2,d_3,k^{2g-3})]=\sum c_i [V_i]
\end{equation*}
for $c_i>0$, $V_i$ irreducible with class not proportional to $[(\varphi_j)_*\overline{\PPP}(d_1,d_2,d_3,k^{2g-3})]$. Let $\pi_m:\overline{\mathcal{M}}_{g,2g-j}\longrightarrow \overline{\mathcal{M}}_{g,2g-j-1}$ forget the $m$th point. Pushing forward yields 
\begin{equation*}
(\pi_m)_*[(\varphi_j)_*\overline{\PPP}^k(d_1,d_2,d_3,k^{2g-3})]= [(\varphi_{j+1})_*\overline{\PPP}^k(d_1,d_2,d_3,k^{2g-3})]    =\sum c_i (\pi_k)_*[V_i]
\end{equation*}
for each $m=4,...,2g-j$ for $g\geq 3$. Without loss of generality, assume that $d_3=k$ when $g=3$. Then the equation will hold in the $g=3$ case for $m=3,\dots, 6-j$. 

Now fix $m$, as the LHS is non-zero there is at least one $V_i$ such that $(\pi_m)_*[V_i]$ is non-zero and as the LHS is extremal, $(\pi_m)_*[V_i]$ is necessarily a positive multiple of the rigid cycle $[(\varphi_{j+1})_*\overline{\PPP}^k(d_1,d_2,d_3,k^{2g-3})]$. Hence $V_i$ must be supported on 
\begin{equation*}
(\pi_m)^{-1}(\varphi_{j+1}{}_*\overline{\PPP}^k(d_1,d_2,d_3,k^{2g-3}))
\end{equation*}
and further, $(\pi_{m'})_*[V_i]$ is non-zero for any other $m'=4,...,2g-j$ for $g\geq 4$ or $m'=3,\dots 6-j$ for $g=3$. 
This argument for each $m'$ yields $V_i$ is supported in the intersection of $(\pi_m)^{-1}(\varphi_{j+1}{}_*\overline{\PPP}^k(d_1,d_2,d_3,k^{2g-3}) )$ for $m=4,...,2g-j$ for $g\geq 4$ or $m=3,\dots 6-j$ for $g=3$. A general element of $V_i$ is hence of the form $[C,p_1,...,p_{2g-j}]\in \mathcal{M}_{g,2g-j}$ with
\begin{equation*}
d_1p_1+d_2p_2+d_3p_3+\sum_{i=4,i\ne m}^{2g-j}kp_i +\sum_{i=1}^{j+1}kq_i \sim kK_C
\end{equation*}
for some $q_i$ with $m=4,...,2g-j$. But this implies that for $g\geq4$ the $p_i$ for $i=4,...,2g-j$ are all pairwise distinct and hence distinct. Similarly for $g=3$ the $p_i$ for $i=3,\dots,6-j$ are all distinct. 

Hence $V_i$ is supported on $(\varphi_{j})_*\overline{\PPP}^k(d_1,d_2,d_3,k^{2g-3})$ and $[V_i]$ is a positive multiple of $[(\varphi_j)_*\overline{\PPP}^k(d_1,d_2,d_3,k^{2g-3})]$ providing a contradiction. 

The base case for the inductive argument is the divisorial case $j=g-1$ presented in Theorem~\ref{thm:extremal}.
\end{proof}

To extend Proposition~\ref{mero1} to all cases of Theorem~\ref{fullextremal} we will need to rule out some new candidates for the cycle obtained by the above process. We will do this via intersection theory, here we provide two lemmas giving explicit values of intersection numbers we'll require.

\begin{lem}\label{intg2}
For $k\geq2$, $\mu=(-h,k^2,h)$ with $k\nmid h$, in $\overline{\mathcal{M}}_{2,3}$
\begin{equation*}
B^{3,k}_{\mu,1,-1}\cdot\delta_{0:\{2,3\}}=8k^2h^2.
\end{equation*}
\end{lem}

\begin{proof}
We need to enumerate the points $p_1$ and $p_2$ such that
\begin{equation*}
-hp_1+2kp_2+hq_1+q_2-q_3\sim kK_C
\end{equation*}
with $p_1\ne p_2$, $p_i\ne q_j$ and any limits that may occur with these points colliding that will satisfy the global $k$-residue condition.

Consider the map
\begin{eqnarray*}
\begin{array}{cccccc}
f_{h,-2k}:&C\times C&\longrightarrow &\Pic^{h-2k}(C)\\
&(p_1,p_2)&\longmapsto&\OO_C(hp_1-2kp_2).
\end{array}
\end{eqnarray*}
discussed in~[\cite{Mullane4}, \S 2.5]. Analysing the fibre of this map above $hq_1-q_2+q_3-kK_C\in \Pic^{h-2k}(C)$ will provide us with the solutions of interest. By~[\cite{Mullane4}, Proposition 2.4] for $h\geq 1$ and $h\ne 2k$ this map is finite of degree $8k^2h^2$, simply ramified along the diagonal $\Delta$ and the locus of pairs of points that are conjugate under the hyperelliptic involution denoted $I$. For $h=2k$ this map is generically finite of degree $8k^2h^2=32k^2$, contracts $\Delta$ and is simply ramified along $I$. Further, for general $q_i$ the fibre will contain no solutions where $p_1$ and $p_2$ coincide with each other or any of the $q_i$. 
\end{proof}

\begin{lem}\label{intg3}
For $k\geq 2$, $\mu=(d_2,d_3,k^3,d_1)$ with $k\nmid d_1$ and $d_1+d_2+d_3=k$, in $\overline{\mathcal{M}}_{3,4}$
\begin{equation*}
B^{4,k}_{\mu,1,-1}\cdot \delta_{0:\{3,4\}}=24k^2d_2^2d_3^2.
\end{equation*}
\end{lem}

\begin{proof}
We need to ennumerate the points $p_1,p_2,p_3$ such that for fixed general $q_i$,
\begin{equation*}
d_2p_1+d_3p_2+2kp_3+kq_1+d_1q_2+q_3-q_4\sim kK_C
\end{equation*}
with $p_i\ne p_j$ for $i\ne j$ and $p_i\ne q_j$ and any limits that may occur with these points colliding that will satisfy the global $k$-residue condition.

Consider the map
\begin{eqnarray*}
\begin{array}{cccccc}
f:C^{3}&\longrightarrow &\Pic^{3k-d_1}(C)\\
(p_1,p_2,p_3)&\longmapsto&O_C(d_2p_1+d_3p_2+2kp_3).
\end{array}
\end{eqnarray*}
The fibre of this map above $kK_C(-kq_1-d_1q_2-q_3+q_4)\in\Pic^{3k-d_1}(C)$ will give us the solutions of interest. Take a general point $e\in C$ and consider the isomorphism
\begin{eqnarray*}
\begin{array}{cccccc}
h:\Pic^{3k-d_1}(C)&\longrightarrow &J(C)\\
L&\longmapsto&L\otimes\OO_C(-de).
\end{array}
\end{eqnarray*}
Now let $F=h\circ f$, then $\deg F=\deg f$. Observe
\begin{equation*}
F(p_1,p_2,p_3)=\OO_C(d_2(p_1-e)+d_3(p_2-e)+2k(p_3-e)).
\end{equation*}
Let $\Theta$ be the fundamental class of the theta divisor in $J(C)$. By \cite{ACGH} \S1.5 we have
\begin{equation*}
\deg \Theta^g=g!=6
\end{equation*}
and the dual of the locus of $\OO_C(m(x-e))$ for varying $x\in C$ has class $m^2\Theta$ in $J(C)$. Hence
\begin{eqnarray*}
\deg F&=&\deg F_*F^*([\OO_C])\\
&=&\deg \left(d_2^2\Theta \cdot d_3^2\Theta\cdot (2k)^2\Theta   \right)\\
&=&24k^2d_2^2d_3^2.
\end{eqnarray*}
As we have chosen the $q_i$ general, the general fibre will contain no points where the $p_i$ coincide with each other or with the $q_i$. Hence we have found all solutions.
\end{proof}
With the above lemmas we proceed to prove the remaining cases of Theorem~\ref{fullextremal}.
\begin{prop}\label{mero2}
For $g=3$ and $j=0,1,2$ if $\PPP^k(d_1,d_2,d_3,k^{3})$ is irreducible, the cycle $[\varphi_j{}_*\overline{\PPP}^k(d_1,d_2,d_3,k^{3})]$ is rigid and extremal in $\mbox{Eff}^{\hspace{0.1cm}3-j}(\overline{\mathcal{M}}_{g,6-j})$, where $\varphi_j:\overline{\mathcal{M}}_{3,6}\longrightarrow\overline{\mathcal{M}}_{3,6-j}$ forgets the last $j$ points with $d_1+d_2+d_3=k$ and $d_i\ne 0$.
\end{prop}

\begin{proof} The case where some $d_i=k$ is covered by Proposition~\ref{mero1}. Assume $d_i\ne k$, we proceed by induction. 
The argument used in the proof of Proposition~\ref{mero1} shows that if $[(\varphi_{1})_*\overline{\PPP}^k(d_1,d_2,d_3,k^{3})]$ is rigid and extremal then $[\overline{\PPP}^k(d_1,d_2,d_3,k^{3})]$ is rigid and extremal.

In the remaining case $j=1$, $V_i$ is supported in the intersection of $(\pi_m)^{-1}(\varphi_{2}{}_*\overline{\PPP}(d_1,d_2,d_3,1^{3}) )$ for $m=4,5$. In this case there are two possible candidates for where the irreducible cycle $V_i$ is supported. By assumption the cycle $V_i$ is not supported on either $\varphi_1{}_*\overline{\PPP}^k(d_1,d_2,d_3,k^{3})$ so it must be supported on $X=\alpha_*D^{4,k}_{d_1,d_2,d_3,k^3}$ where
\begin{eqnarray*}
\begin{array}{cccc}
\alpha:&\overline{\mathcal{M}}_{3,4}&\rightarrow& \overline{\mathcal{M}}_{3,5}\\
&[C,p_1,p_2,p_3,y]&\mapsto&[C\bigcup_{y=y'}Y,p_1,\dots,p_5].
\end{array}
\end{eqnarray*} 
where $[Y,y',p_4,p_5]$ is a rational curve marked at three distinct points. Then $X$ is irreducible if $\overline{\PPP}^k(d_1,d_2,d_3,k^{3})$ is irreducible. Hence $[V_i]$ is proportional to $[X]$ and
\begin{equation*}
\pi_1{}_*[V_i]=  e\delta_{0:\{3,4\}}
\end{equation*}
for some $e>0$.  As the cycle $[(\varphi_1)_*\overline{\PPP}^k(d_1,d_2,d_3,k^{3})]-c_i[V_i]$ is an effective codimension two cycle in $\overline{\mathcal{M}}_{3,5}$, by pushing down under the morphism that forgets the first marked point we obtain the effective divisor class 
\begin{equation*}
\pi_1{}_*\left([(\varphi_j)_*\overline{\PPP}^k(d_1,d_2,d_3,k^{3})]-c_i[V_i]\right)=D^{4,k}_{d_2,d_3,k^3,d_1}-c_i e\delta_{0:\{3,4\}}.
\end{equation*}
However, by Theorem~\ref{thm:boundary} and Lemma~\ref{intg3} we observe
\begin{equation*}
B^{4,k}_{\mu,1,-1}\cdot (D^{4,k}_{\mu}-c_i e\delta_{0:\{3,4\}})=0-24k^2d_2^2d_3^2c_ie<0,
\end{equation*}
for $\mu=(d_2,d_3,k^3,d_1)$, which contradicts the assertion of Proposition~\ref{moving} that $B^{4,k}_{\mu,1,-1}$ is a moving curve. Hence $V_i$ is not supported on $X$ and must be supported on $\varphi_{1}{}_*\overline{\PPP}^k(d_1,d_2,d_3,k^{3})$ contradicting our assumption of an effective decomposition. Hence $[\varphi_{1}{}_*\overline{\PPP}^k(d_1,d_2,d_3,k^{3})]$ is rigid and extremal if $[\varphi_{2}{}_*\overline{\PPP}^k(d_1,d_2,d_3,k^{3})]$ is rigid and extremal. The base case for the inductive argument is the divisorial case $j=2$ presented in Theorem~\ref{thm:extremal}.
\end{proof}

\begin{prop}\label{mero3}
For $g= 2$ and $\gcd(h,k)=1$ if $\PPP^k(h,-h,k,k)$ is irreducible, the cycle $[\overline{\PPP}^k(h,-h,k,k)]$ is rigid and extremal in $\mbox{Eff}^{\hspace{0.1cm}2}(\overline{\mathcal{M}}_{2,4})$.
\end{prop}

\begin{proof}
$[\pi_m{}_*\overline{\PPP}(h,-h,k,k)]$ is rigid and extremal for $m=3,4$ by Theorem~\ref{thm:extremal}. 
If $[\overline{\PPP}^k(h,-h,k,k)]$ is not extremal then it can be expressed as
\begin{equation*}
[\overline{\PPP}(h,-h,1,1)]=\sum c_i [V_i]
\end{equation*}
for $c_i>0$, $V_i$ irreducible with class not proportional to $[\overline{\PPP}^k(h,-h,k,k)]$. Hence by the same argument in the proof of Proposition~\ref{mero2} we obtain that  some $V_i$ must be supported on the intersection 
\begin{equation*}
\pi_3^{-1}(\pi_3{}_*\overline{\PPP}^k(h,-h,k,k))\cap \pi_4^{-1}(\pi_4{}_*\overline{\PPP}^k(h,-h,k,k))
\end{equation*}
which has two irreducible components. However, as $V_i$ is by assumption, not supported on $\overline{\PPP}^k(h,-h,k,k)$ it must be supported on $X=\alpha_*D^{3,k}_{h,-h,k^2}$ where
\begin{eqnarray*}
\begin{array}{cccc}
\alpha:&\overline{\mathcal{M}}_{2,3}&\rightarrow& \overline{\mathcal{M}}_{2,4}\\
&[C,p_1,p_2,y]&\mapsto&[C\bigcup_{y=y'}Y,p_1,\dots,p_4].
\end{array}
\end{eqnarray*} 
where $[Y,y',p_3,p_4]$ is a rational curve marked at three distinct points and 
\begin{equation*}
D^{3,k}_{h,-h,k^2}=\varphi_1{}_*\overline{\PPP}^k(h,-h,k,k)=\pi_4{}_*\overline{\PPP}^k(h,-h,k,k).
\end{equation*}
The irreducibility of $X$ follows from the assumed irreducibility of $\overline{\PPP}^k(h,-h,k,k)$. Hence if $V_i$ is supported on $X$ then $[V_i]$ is proportional to $[X]$ and  
\begin{equation*}
\pi_1{}_*[V_i]=  e\delta_{0:\{2,3\}}
\end{equation*}
for some $e>0$.  

As the cycle $[\overline{\PPP}^k(h,-h,k,k)]-c_i[V_i]$ is effective, by pushing down under the morphism that forgets the first marked point we obtain the effective class 
\begin{equation*}
\pi_1{}_*\left([\overline{\PPP}^k(h,-h,k,k))]-c_i[V_i]\right)=D^{3,k}_{-h,k,k,h}-c_i e\delta_{0:\{2,3\}}.
\end{equation*}
However, by Proposition~\ref{thm:boundary} and Lemma~\ref{intg2} 
\begin{equation*}
B^{3,k}_{\mu,1,-1}\cdot (D^{3,k}_{-h,k,k,h}-c_i e\delta_{0:\{2,3\}})=0-8k^2h^2ec_i<0,
\end{equation*}
for $\mu=(-h,k,k,h)$ which contradicts the assertion of Proposition~\ref{moving} that $B^{3,k}_{\mu,1,-1}$ is a moving curve. Hence $V_i$ is not supported on $X$ and must be supported on $\overline{\PPP}^k(h,-h,k,k)$ providing a contradiction with the given effective decomposition. Hence $[\overline{\PPP}^k(h,-h,k,k)]$ is rigid and extremal.
\end{proof}


\bibliographystyle{plain}
\bibliography{base}
\end{document}